\documentclass{article}

\usepackage{geometry}
\usepackage{amsmath}
\usepackage{amssymb}
\usepackage{amsthm}
\usepackage{MnSymbol}
\usepackage[numbers]{natbib}
\usepackage{quiver}
\usepackage[colorlinks=true]{hyperref}
\hypersetup{allcolors=[rgb]{0.1,0.1,0.8}}
\usepackage{parskip}
\usepackage{comment}

\newenvironment{ak}{\begin{color}{magenta}}{\end{color}}

\newcommand{\Sup}{\mathsf{Sup}}

\newcommand{\Rel}{\mathsf{Rel}}
\newcommand{\Spa}{\mathsf{Spa}}
\newcommand{\cont}{{\bf I}}
\newcommand{\bulletop}{\mathbin{\bullet}}
\newcommand{\val}[1]{[\![{#1}]\!]}
\newcommand{\descr}[1]{(\![{#1}]\!)}
\renewcommand{\phi}{\varphi}

\newtheorem{theorem}{Theorem}[section]
\newtheorem{lemma}[theorem]{Lemma}
\theoremstyle{definition}
\newtheorem{proposition}[theorem]{Proposition}
\newtheorem{corollary}[theorem]{Corollary}
\newtheorem{definition}[theorem]{Definition}
\newtheorem{example}[theorem]{Example}
\newtheorem{remark}[theorem]{Remark}

\title{Canonical Extensions of Quantale-Enriched Categories}

\author{Alexander Kurz\thanks{Chapman University, Orange, California, USA} \and 
Apostolos Tzimoulis\thanks{University of Luxembourg, Luxembourg}}

\date{}

\begin{document}

\maketitle

\begin{abstract}
Drawing on well-known results from the theory of canonical extensions and the theory of categories enriched over a quantale, we define canonical extensions of quantale-enriched categories and establish their basic properties.
\end{abstract}

\tableofcontents

\section{Introduction}

Drawing on well-known results from the theory of canonical extensions and the theory of categories enriched over a quantale 
$$(\Omega,\sqsubseteq,\bigsqcup,e,\cdot,)$$
we define canoncial extensions of quantale-enriched categories and establish their basic properties. Elements of $\Omega$ can be understood as distances or weights or truth-values.

Canonical extensions arise as certain MacNeille completions. MacNeille completions, in turn, are given by the fixed points of the MacNeille-Birkhoff-Isbell-Lawvere adjunction 
% https://q.uiver.app/#q=WzAsMixbMCwwLCJcXG1hdGhjYWwgRCBYIl0sWzIsMCwiXFxtYXRoY2FsIFUgQSJdLFswLDEsInBeXFx1cGFycm93IiwwLHsiY3VydmUiOi0yfV0sWzEsMCwicF5cXGRvd25hcnJvdyIsMCx7ImN1cnZlIjotMn1dXQ==
\[
\begin{tikzcd}
	{\mathcal{D}X} && {\mathcal{U}A}
	\arrow["{\phi\ \mapsto\ \phi\blacktriangleright \cont}",curve={height=-16pt}, from=1-1, to=1-3]
	\arrow["{\psi\ \mapsto\ \cont\blacktriangleleft\psi}",curve={height=-16pt}, from=1-3, to=1-1]
	\arrow[phantom, "\bot", from=1-1, to=1-3]
\end{tikzcd}
\]
between ``weighted downsets'' $\phi$ on $X$ and ``weighted upsets'' $\psi$ on $A$ induced by a binary relation
$$\cont :X\looparrowright A$$
which is an $\Omega$-valued relation 
$$\cont: X \times A\to \Omega$$
on the product of the enriched categories $X$ and $A$. The lattice theoretic results on MacNeille completions and canonical extensions are recovered in the case where the quantale $\Omega$ is the familiar set $2=\{0<1\}$ of truth-values.

We develop our work in a language that stays as faithful as possible to \emph{both} lattice theory (as in the work of Dunn, Gehrke and Palmigiano \cite{Dunn-etal:canonical-extensions} on the canonical extensions of posets) and to enriched category theory (as in the work of Stubbe~\cite{Stubbe05a} on quantaloid enriched categories). 
This does require some compromise. 
For example, since we do not assume that the quantale $\Omega$ is commutative, we need two implications (residuals), written as $\rhd$ and $\lhd$. 
Because of the associativity of the quantale multiplication $(a \cdot b)\cdot c = a\cdot (b\cdot c)$ the residuals satisfy
$$ (a\rhd b)\lhd c = a \rhd (b\lhd c)$$
We systematically extend quantale notation to a  ``distributor calculus''~\cite{Stubbe05a}, where this equation becomes 
$$ (\phi \blacktriangleright \cont)\blacktriangleleft \psi = \phi \blacktriangleright (\cont\blacktriangleleft \psi).$$
and is nothing but the MacNeille-Birkhoff-Isbell-Lawvere adjunction 
$$-\blacktriangleright \cont \ \dashv \ 
 \cont\blacktriangleleft -. $$
This ``blacktriangle calculus'' provides a useful bridge between quantale-enriched category theory on the one hand and lattice theory, algebraic logic and proof theory on the other hand. 

The reason for paying attention to the language in which we formulate our results is not only to make our work accessible to category theorists and to lattice theorists: In our own ongoing work on applications to logic, discussed below, we need to have easy access to general category theoretic results and to an algebraic language in the style of lattice theory and logic. 
%As we will discuss in the conclusion, 
This approach also sheds new light on known results in lattice theory. 

In the remainder of this section, we discuss applications and related work.

\subsection{Applications to Logic}%Categorization, Decision Theory and AI ... 
The present article is  situated in a line of research which develops and studies logical formalisms that describe and reason about categorization systems (with applications to decision theory and AI) arising across disciplines. The main methodological tool of this endeavor is based on Wille's formal concept analysis (FCA) \cite{Ganter-Wille}. In  \cite{CFPPTW16} the epistemic logic of categories and concepts was introduced, based on a general framework of non-distributive logics developed in \cite{CoPa:non-dist}. In \cite{belohlavek,Pavlovic12}, formal concept analysis was lifted to the fuzzy context, where relations that take value over commutative unital quantales and their corresponding fuzzy Galois connection were studied. Recently, combining these approaches, correspondence theory for  fuzzy non-distributive modal logics is developed in \cite{CONRADIE2024108892}; in \cite{BOERSMA2024114196} such logical frameworks have been used for developing unsupervised learning algorithms for outlier detection that also provide explanations of their results. In this context, the present article is a stepping stone for generalizing algorithmic correspondence \cite{CONRADIE2010319,CoPa:non-dist} results from the algebraic to the quantale-enriched setting. 

\subsection{Related Work} 

While the specific motivation of our work are the applications to logic discussed above, we build on insights and results from various areas of mathematics which we sketch out now.

\paragraph{Lattice Theory} In Birkhoff's monograph \emph{Lattice Theory} \cite{Birkhoff48}, under the heading \emph{Polarity}, one finds the definition of what we call the MacNeille completion of a binary relation 
$\cont: X\looparrowright A$
in the special case $\Omega=2$, that is, where the quantale $\Omega$ is the familiar set $2=\{0<1\}$ of truth-values. For a modern introduction to lattice theory see Davey and Priestley~\cite{Davey-Priestley}.

\paragraph{Formal Context Analysis} As already shown by Birkhoff, relations $\cont: X\looparrowright A$ can be considered (via their MacNeille completions) as representations of complete lattices. The Formal Context Analysis (FCA) of Ganter and Wille \cite{Ganter-Wille} takes this as a starting point for a theory of data-analysis in which the `formal context' or `context' or `incidence relation', $\cont$ relates `objects' in $X$ with `attributes' in $A$.  While we will be interested in the \href{https://upriss.github.io/fca/fca.html}{applications of FCA} in future work, most relevant for our work here are theoretical investigations into generalizations of FCA to many-valued (or fuzzy or quantitative) context analysis as in the work of Belohlavek~\cite{belohlavek} and Pavlovic~\cite{Pavlovic12}.

\paragraph{Canonical Extensions} Canonical extensions of Boolean algebras were introduced by Jonsson and Tarski \cite{Jonsson-Tarski:1,Jonsson-Tarski:2} in order to study relation algebras and, more generally, modal algebras (Boolean algebras with operators). In particular, they pioneered the use of canonical extensions to prove completeness of modal logics via representation theorems generalizing Stone duality for Boolean algebras. This line of research was extended by Gehrke and Jonsson \cite{Gehrke-Jonsson} to distributive lattices and by Dunn, Gehrke and Palmigiano~\cite{Dunn-etal:canonical-extensions} to lattices and posets, that is, in our terms, to the case $\Omega=2$.

\paragraph{Substructural Logics, Algebraic Logic} On the one hand, quantales can be seen as particular residuated lattices, which constitute the algebraic semantics of \emph{substructural logics}, see Galatos, Jipsen, Kowalski and Ono \cite{residuated-lattices}. Here, substructural refers to the omission of (some of) the rules of exchange, weakening and contraction in the corresponding sequent calcului. On the other hand, MacNeille completions of relations allow us to give syntactic representations of \emph{non-distributive} logics. Our work will allow us to generalize duality-based results in modal logic to non-distributive substructural logics. 

\paragraph{Category Theory - Lawvere Metric Spaces} Like so many others, we are inspired by Lawvere's seminal \emph{Metric spaces, generalized logic and closed categories}~\cite{Lawvere:metric}. Lawvere's article suggests that it should be possible to produce a uniform account of a considerable body of results on many-valued logics parametrically in the quantale $\Omega$ of truth-values and here we continue previous work in this line of research \cite{BKPV:cover-modality,Babus-Kurz,BalanKV,Balan-Kurz:distributivity}. Our running example of automata as categories enriched over a (non-commutative) quantale in Section~\ref{sec:preliminaries}, due to Betti and collaborators \cite{Betti:automi,Betti-Kasangian,Betti-Walters}, was also inspired by \cite{Lawvere:metric}. Another indirect influence of \cite{Lawvere:metric} on our work is via Willerton's \cite{Willerton:tight-spans} which gives a detailed analysis of the MacNeille completion of Lawvere metric spaces. Hofmann and Stubbe~\cite{Hofmann-Stubbe} consider the generalisation to partial metrics. The MacNeille completion of quantale-enriched categories has been studied in Shen and Zhang~\cite{Shen-Zhang}, Garner \cite{Garner:topological}, and Fujii \cite{Fujii}. The latter extends the characterization by Banaschewski and Bruns~\cite{Banaschewski:macneille} of the MacNeille completion as the injective hull from the case of $\Omega=2$ to all quantales.

\paragraph{Category Theory - Isbell Conjugacy} Isbell  \cite{Isbell:adequate-subcategories} generalized the MacNeille-Birkhoff adjunction from orders to categories and Lawvere \cite{Lawvere:isbell} called it the Isbell conjugacy. Isbell also generalized the MacNeille completion to categories (then often called the reflexive completion). Avery and Leinster~\cite{Avery-Leinster} investigate it in  detail. Pavlovic and Hughes \cite{Pavlovic-Hughes} and then Ferrer~\cite{Ferrer:phd} give a generalization of the MacNeille completion to categories based on tight limits. In a different direction, Kurz and Rosicky \cite{Kurz-Rosicky:logics} used the Isbell conjugacy to generalize the Jonsson-Tarski representation theorem for Boolean algebras with operators to modal algebras that dualize coalgebras for set-functors.

\paragraph{Domain Theory} Lawvere metric spaces \cite{Lawvere:metric} are asymetric and can represent both order-theoretic and metric structure. Wagner~\cite{Wagner:PhD,Wagner97} and Rutten and collaborators \cite{Rutten96,Rutten98,Bonsangue98} used this observation to unify order-theoretic and metric domain theory. This work shows that quantale-enriched categories provide a wide range of different models of computations.

\paragraph{Coalgebraic Modal Logic} Coalgebraic modal logic \cite{Moss:coalgebraic-logic,KurzPSV} has been successful in generalizing Abramsky's Domain Theory in Logical Form~\cite{Abramsky:dtlf} from particular type constructors to a rather general theory parameterized by a functor on a suitable category such as sets or posets. While there has been work on further parameterizing coalgebraic logic by a quantale \cite{BKPV:cover-modality,Babus-Kurz,BalanKV,Balan-Kurz:distributivity}, it remains an open question how far this program can be pushed.

%\subsection{Table of Contents} We start with two sections on Preliminaries (Section~\ref{sec:preliminaries}) and on the Mac Neille Completion of a relation (Section~\ref{sec:macneille}). These sections contain known material collected from various sources such as \cite{Stubbe05a,Garner:topological,Shen-Zhang}. The reason to include a detailed account here is to unify and streamline the notation in such a way that it will faciliate the applications to logic outlined above. Particular attention was paid to establishing a notation that allows a smooth transition from laws on the level of quantales to laws on the level of enriched categories. 

\subsection{Acknowledgements} This paper is part of a larger project in collaboration with Giuseppe Greco and Brandon Laing on the algebraic semantics and proof theory of many-valued non-classical substructural logics. We are also grateful to Peter Jipsen, Drew Moshier and Alessandra Palmigiano for helpful discussions.

\section{Preliminaries}\label{sec:preliminaries}

We review quantales, quantale-enriched categories (called quantale spaces here), quantale-valued relations and weighted (co)limits. We discuss in some detail four running examples and are careful to introduce notation that will help with later calculations. For more details, we recommend Stubbe~\cite{Stubbe05a,Stubbe:introduction}.\footnote{We write Stubbe's $g\circ f$ as  $f\cdot g$ reading $f\cdot g$ as ``first $f$, then $g$".}

\subsection{Quantales}

We denote by $\Sup$ the monoidal closed category of complete join semilattice (sup-lattices). A quantale is a monoid in the monoidal category of sup-lattices or also a one-object $\Sup$-enriched category. Explicitly, a quantale 
$$(\Omega,\sqsubseteq, \bigsqcup,e,\cdot,)$$
is a complete join semilattice $(\Omega,\sqsubseteq,\bigsqcup)$ and a monoid $(\Omega,e,\cdot)$ in which multiplication distributes over joins. We write top as $\top$ and bottom as $\bot$. Since $\Omega$ is complete it also has meets $\bigsqcap$.\footnote{The order of operations is defined such that in $\bigsqcap_x(\alpha(x)\rhd\beta(x))$ the parentheses can be omitted.} Since multiplication preserves joins in each argument,  multiplication has a left-residual $\lhd$ and a right-residual $\rhd$ defined as
$$b\,\sqsubseteq\, a\rhd c 
\quad \Leftrightarrow \quad
a \cdot b \,\sqsubseteq\, c 
\quad \Leftrightarrow \quad
a\,\sqsubseteq\, c\lhd b $$
\noindent In a commutative quantale, we have $a\rhd c = c\lhd a$.

\begin{example}\label{exle:quantales} 
\begin{enumerate}
\item The \emph{two-chain} $2=\{0\sqsubseteq 1\}$ is a commutative quantale in which multiplication and meet coincide. The residual is implication.

\item The \emph{Lawvere quantale} $[0,\infty]$ is a subset of the extended real numbers \cite{Lawvere:metric}. It is ordered by $\ge$ with top $\top=0$ and has $+$ as multiplication. The residual is truncated minus $a\lhd b=a\dotminus b$. 

\item The \emph{similarity quantale} $\Omega=\{0, 1,  2,   \ldots \infty\}$ is ordered like the extended natural numbers, has $\min(-,-)$ as multiplication  and $\top=\infty$ as the neutral element. We interpret the elements of $\Omega$ as a measure of similarity. The residual $m\rhd n$ is given by ${\sf if~} m\le n {\sf ~then~ } \infty {\sf ~else~} n$.

\item The \emph{quantale  of languages} $\mathcal P(\Sigma^\ast)$ is given with respect to a set $\Sigma$ (called the alphabet) and has as elements subsets of $\Sigma^\ast$ (the set of finite words over $\Sigma$) \cite{Abramsky-Vickers:quantales}. The order is given by inclusion. Multiplication is defined via $L\cdot L'=\{vw \mid v\in L, w\in L'\}$ where $vw$ denotes the concatenation of the words $v$ and $w$. The residuals are given by $L\rhd M = \{w\in \Sigma^\ast \mid \forall v\in L\,.\, vw\in M\}$ and $M\lhd L = \{w\in \Sigma^\ast\mid \forall v\in L \,.\, wv\in M\}$. We write $\epsilon$ for the empty word and usually abbreviate a singleton set---such as the identity of multiplication $\{\epsilon\}$---by its element. 
\end{enumerate}
To be continued in Example~\ref{exle:quantale-spaces}. 
\end{example}

%\begin{remark}
\subsection{Quantale Laws}\label{sec:quantale-laws}
%We list some useful facts. 
The counits of the residuals are 
\begin{gather*}
r\cdot (r\rhd s) \sqsubseteq s\\
s\sqsupseteq (s\lhd r) \cdot r
\end{gather*}
For all $r,s,t\in\Omega$ we have 
\begin{gather*}
r\rhd(s\lhd t) =  (r\rhd s) \lhd t\\
(r\rhd s)\cdot(s\rhd t)\sqsubseteq r\rhd t\\
(t\lhd s)\cdot(s\lhd r)\sqsubseteq t\lhd r\\
e\sqsubseteq (r\rhd r) \quad\quad e\sqsubseteq (r\lhd r)
\end{gather*}
For all $r\in\Omega$ there is an adjunction between partial orders
$$(-\rhd r)\ \dashv \ (r\lhd -) : \Omega^\partial\to\Omega$$
%\end{remark}
due to $\Omega^\partial(a\rhd r,b)=\Omega(a,r\lhd b)$.

\subsection{Duality}
The notation $$\Omega^\partial$$ refers to the poset obtained from $\Omega$ by reversing the order while $$\Omega^o$$ is obtained from $\Omega$ by reversing the multiplication. Considering $\Omega$ as a one-object order-enriched category, $\Omega^o$ is the category obtained from $\Omega$ by reversing 1-cells.

\subsection{Quantale Spaces}

We call a category enriched over a quantale $\Omega$ a \emph{quantale space}, or, $\Omega$-space. Explicitely, an $\Omega$-space $X$ consists of set $X$ together with a function $X(-,-):X\times X\to\Omega$ satisfying
\begin{align*}
    e &\sqsubseteq X(x,x),\\
    X(x,y)\cdot X(y,z) & \sqsubseteq X(x,z).
\end{align*}
A morphism of quantale spaces is a functor of quantale-enriched categories. Explicitely, a function $f:X\to Y$ is a morphism of quantale spaces, also known as a functor, if 
$$X(x,x')\sqsubseteq Y(fx,fx').$$
\begin{definition}
    Every quantale space $X$ has an underlying order
$$x\le y \stackrel{\rm def}\Longleftrightarrow e\sqsubseteq X(x,y).$$
Given $f,g:X\to Y$ the order on functors is defined by $f\le g$ if $fx\le gx$ for all $x\in X$. A space is called \emph{skeletal} if its underlying order is anti-symmetric.
\end{definition}

\begin{definition}
    $\Spa(\Omega)$ is the order-enriched category of quantale spaces and quantale space morphisms.
\end{definition}

\begin{example} \label{exle:quantale-spaces} We continue from Example~\ref{exle:quantales}.
\begin{enumerate}
\item $\Spa(2)$ is isomorphic to the category of preorders.
\item $\Spa([0,\infty])$ is the category of Lawvere metric spaces, also known as a generalized metric space \cite{Lawvere:metric}. For a Lawvere metric space (LMS) $X$, the `distance' $X(x,y)$ satsifies $X(x,x)=0$ and the triangle inequality $X(x,y)+X(y,z) \ge X(x,z)$. Every metric space is a Lawvere metric space. But an LMS does not need to have symmetric distance:
\begin{enumerate}
    \item $\Omega=[0,\infty]$ is an LMS with $\Omega(x,y)=y\dotminus x$. Note that the order $\sqsubseteq$ on $[0,\infty]$ satisfies $x\sqsubseteq y \Leftrightarrow 0=y\dotminus x$.
    \item \label{exle:real} The real numbers $\mathbb R$ with their natural order are an LMS given by $\mathbb R(x,y)=x\dotminus y$. Note that $x\le y$ in the reals iff $0=x\dotminus y$. 
\end{enumerate}
\item Let $\Omega$ be the similarity quantale. We call an $\Omega$-space a similarity space. The set $\Sigma^\ast$ of finite words over an alphabet $\Sigma$ is a similarity space defined by $\Sigma^\ast(v,w)=\infty$ if $v$ is a prefix of $w$ and otherwise $\Sigma^\ast(v,w)$ is the length of the longest common prefix of $v,w$. The same definition also turns the set of finite and infinite words over $\Sigma$ into a similarity space.

Similarity spaces can be embedded into generalized ultrametric spaces, which have been studied from the point of view of enriched category theory in Rutten~\cite{Rutten96,Rutten98}.

\item A quantale space over $\mathcal P(\Sigma^\ast)$ is a generalized non-deterministic automaton (without designated initial and final states) \cite{Betti:automi,Betti-Kasangian,Rosenthal:automata}.
\end{enumerate}
To be continued in Example~\ref{exle:omega}.
\end{example}

\subsection{Truth-Values vs Distances}
As can be seen from the examples, the elements of the quantale can be interpreted as \emph{truth values} or as \emph{distances}. For example, $[0,\infty]$ with $0$ as top and $\infty$ as bottom is naturally interpreted as a quantale of distances. On the other hand, there is an isomorphism of quantales $[0,\infty]\to [0,1]$ defined by $x\mapsto e^{-x}$ mapping $0$ to $1$ and $\infty$ to $0$. While both quantales give rise to isomorphic quantale enriched categories, the second quantale is more naturally interpreted as a quantale of truth values.

In this example, in the \emph{truth-value interpretation} we have 
$X(a,b)\cdot X(b,c) \le X(a,c)$ while in the \emph{distance interpretation} we have $X(a,b)+X(b,c)\ge X(a,c)$, also known as the triangle inequality of metric spaces. Similarly, $\bigsqcap$ is conjunction in the truth-value interpretation and supremum in the distance interpretation.

In the following, for $\Omega$-spaces $X$, we often find it convenient to speak of $X(x,y)$ as \emph{the distance from $x$ to $y$} without making assumptions on $\Omega$ or implying that we favour the distance interpretation over the truth-value interpretation.

\subsection{Duality} Dualising the order of the multiplication of $\Omega$, $X$ is an $\Omega^o$-space if it satisfies
$$X(y,z)\cdot X(x,y) \sqsubseteq X(x,z).$$
In addition, one can also dualise the order of $\Omega$. This leads to 4 different ways in which $\Omega$ itself can be turned into a quantale space $A$: 
\begin{enumerate}
\item $A(a,b)=a\rhd b$
\item $A(a,b)=a\lhd b$
\item $A(a,b)=b\rhd a$
\item $A(a,b)=b\lhd a$
\end{enumerate} 
In cases 1 and 4 the order defined by $a\le b \ \Leftrightarrow\ e\sqsubseteq A(a,b)$ agrees with the order of $\Omega$, in cases 2 and 3 it is the dual. In cases 1 and 2, $A$ is an $\Omega$-space and, in cases 3 and 4, $A$ is an $\Omega^o$-space. The $A$ of case 1 is the $A^o$ of case 3 and the $A$ of case 2 is the $A^o$ of case 4. If $\Omega$ is commutative, case 1 equals case 4 and case 2 equals case 3.
\begin{remark}
    $X$ is an $\Omega$-space iff $X^o$ is an $\Omega^o$-space.
\end{remark}

\begin{example}\label{exle:omega} We continue from Example~\ref{exle:quantale-spaces}.
\begin{enumerate}
    \item The Lawvere quantale is enriched over itself with  $[0,\infty](r,s)=r{\rhd}s = s{\lhd} r = s\dotminus r$.
    
   \item The quantale of all languages $\Omega=\mathcal P(\Sigma^\ast)$ can be enriched to an automaton in 4 different ways. 
   \begin{description}
       \item{1.} $\mathcal P(\Sigma^\ast)(L,M)=L\rhd M$
       \item{2.} $\mathcal P(\Sigma^\ast)(L,M)=L\lhd M$
       \item{3.} $\mathcal P(\Sigma^\ast)(L,M)=M\rhd L$
       \item{4.} $\mathcal P(\Sigma^\ast)(L,M)=M\lhd L$
   \end{description} We interpret case 1 and 2 (enriched over $\Omega$)  as a forward running automaton and case 3 and 4 (enriched over $\Omega^o$) as a backward running automaton. 
   
   We call case 1 the \emph{history automaton of all languages} because we think of a state $L$ as recording its history, starting from the initial state $\epsilon$. The hom $L\rhd M$ contains the words which extend all words in $L$ to a word in $M$. In particular, $\mathcal P(\Sigma^\ast)(\epsilon,L)=\epsilon \rhd L = L$. 
   
   We call case 2 the \emph{prophecy automaton of all languages} because a state $L$ consists of all words that  lead from $L$ to a final state, where we consider a state to be final if it contains $\epsilon$.  In particular, $\mathcal P(\Sigma^\ast)(L,\epsilon)=L\lhd\epsilon =L$ and if $\epsilon\in M$ then $\mathcal P(\Sigma^\ast)(L,M)\subseteq L$.
\end{enumerate}
To be continued in Example~\ref{exle:weighted-limit}.
\end{example}

\subsection{Weighted Relations}\label{sec:relations}
A \emph{quantale-valued relation} (also known as bimodule, profunctor, distributor, or weakening relation) $R:X\looparrowright Y$ between quantale spaces $X$ and $Y$ is a function $X\times Y \to \Omega$ satisfying 
$$X(x',x)\cdot R(x,y)\sqsubseteq R(x',y) 
\quad\quad\quad R(x,y)\cdot Y(y,y')\sqsubseteq R(x,y').$$
We call such an $R$ also a quantale relation or \emph{weighted relation} or metric relation or $\Omega$-relation or just relation. If we want to name the quantale, we call it an $\Omega$-relation. Weighted relations $R : X\looparrowright Y$ and $S:Y\looparrowright Z$ are composed in diagrammatic order according to 
$$(R\bullet S)(x,z)=\bigsqcup_{y\in Y}R(x,y)\cdot S(y,z).$$ 
Given an $\Omega$-space $X$, we use $X$ also to denote the hom $X(-,-):X\looparrowright X$. Homs play the role of identity relations:
\begin{equation}\label{eq:identity}
X\bullet R = R = R\bullet Y.
\end{equation}
We may write $Rxy$ or $xRy$ for $R(x,y)$.

%\begin{ak} we need to say that $\bullet$ is associative\end{ak}

\begin{lemma}\label{lemma: bullet is associative}
    The operation $\bullet$ is associative, that is, $$R\bullet(S\bullet T)=(R\bullet S)\bullet T.$$
\end{lemma}
\begin{proof} This follows from the associativity of the quantale operation and that in a quantale the equations $(\bigsqcup_{i\in I}r_i)\cdot s=\bigsqcup_{i\in I}(r_i\cdot s)$ and $s\cdot(\bigsqcup_{i\in I}r_i)=\bigsqcup_{i\in I}(s\cdot r_i)$ hold:
    \begin{align*}
R\bullet (S\bullet T)(w,z) 
&= \bigsqcup_xR(w,x)\cdot(\bigsqcup_y S(x,y) \cdot T(y,z))\\
&= \bigsqcup_x\bigsqcup_y R(w,x)\cdot(S(x,y)\cdot T(y,z))\\
&= \bigsqcup_x\bigsqcup_y (R(w,x)\cdot S(x,y))\cdot T(y,z)\\
&= \bigsqcup_y\bigsqcup_x (R(w,x)\cdot S(x,y))\cdot T(y,z)\\
&= \bigsqcup_y (\bigsqcup_x R(w,x)\cdot S(x,y))\cdot T(y,z)\\
&= (R\bullet S)\bullet T (w,z).
\end{align*}
\end{proof}

\begin{definition}
$\Omega$-relations form an order-enriched category $\Rel(\Omega)$, or $\Rel$ for short, with composition and identities given as above and the order in $\Rel(X,Y)$ defined by 
$$R\sqsubseteq S \ \ \Longleftrightarrow \ \ \forall x\in X.\forall y\in Y. xRy \sqsubseteq xSy.$$
\end{definition} 

\begin{remark}\label{rmk:blacktriangle}
    $\Rel$ is residuated with, given $R : X\looparrowright Y$, $S:Y\looparrowright Z$, $T:X\looparrowright Z$,
    $$R\bulletop - \ \dashv\  R\blacktriangleright - \quad\quad\quad 
    - \bulletop S\ \dashv \  -\blacktriangleleft S$$
    that is
    $$
    S\sqsubseteq R\blacktriangleright T \ \Leftrightarrow \ 
    R\bullet S\sqsubseteq T \ \Leftrightarrow \ 
    R\sqsubseteq T\blacktriangleleft S.
    $$
    Given this adjunction and Lemma \ref{lemma: bullet is associative} we can lift all the equations and inequations from Section \ref{sec:quantale-laws}. In particular:
    \begin{gather*}
    R\bulletop (R\blacktriangleright T) \,\sqsubseteq \, T
    \quad\quad\quad
    T\sqsubseteq R\blacktriangleright (R\bulletop T) \\[1ex]
    (T\blacktriangleleft S)\bulletop S \,\sqsubseteq\, T
    \quad\quad\quad
    T \,\sqsubseteq\, (T\bulletop R)\blacktriangleleft R\\
    (R \blacktriangleright S)\blacktriangleleft T = R\blacktriangleright (S\blacktriangleleft T)
    \end{gather*}
We know $R\blacktriangleright T$ also as the right Kan extension of $T$ along $R$, while $T\blacktriangleleft S$ is the dual of a right Kan extension obtained from reversing 1-cells but not 2-cells.
% https://q.uiver.app/#q=WzAsNixbMiwyLCJaIl0sWzAsMSwiWSJdLFs0LDIsIlgiXSxbNCwwLCJaIl0sWzYsMSwiWSJdLFsyLDAsIlgiXSxbMSwwLCJSXFxibGFja3RyaWFuZ2xlcmlnaHQgVCIsMl0sWzIsMywiVCJdLFs0LDMsIlMiLDJdLFsyLDQsIlRcXGJsYWNrdHJpYW5nbGVsZWZ0IFMiLDJdLFs1LDAsIlQiXSxbNSwxLCJSIiwyXSxbMSwxMCwiXFxsZSIsMSx7InNob3J0ZW4iOnsic291cmNlIjoyMCwidGFyZ2V0IjoyMH0sInN0eWxlIjp7ImJvZHkiOnsibmFtZSI6Im5vbmUifSwiaGVhZCI6eyJuYW1lIjoiZW5kIn19fV0sWzQsNywiXFxnZSIsMSx7InNob3J0ZW4iOnsic291cmNlIjoyMCwidGFyZ2V0IjoyMH0sInN0eWxlIjp7ImJvZHkiOnsibmFtZSI6Im5vbmUifSwiaGVhZCI6eyJuYW1lIjoiZW5kIn19fV1d
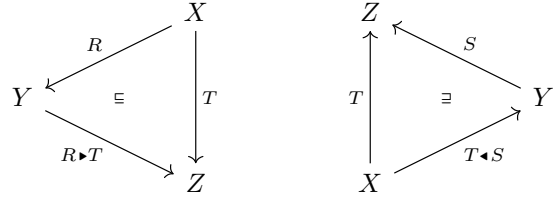
\[\begin{tikzcd}
	&& X && Z \\
	Y &&&&&& Y \\
	&& Z && X
	\arrow["{R\blacktriangleright T}"', from=2-1, to=3-3]
	\arrow[""{name=0, anchor=center, inner sep=0}, "T", from=3-5, to=1-5]
	\arrow["S"', from=2-7, to=1-5]
	\arrow["{T\blacktriangleleft S}"', from=3-5, to=2-7]
	\arrow[""{name=1, anchor=center, inner sep=0}, "T", from=1-3, to=3-3]
	\arrow["R"', from=1-3, to=2-1]
	\arrow["\sqsubseteq"{description}, draw=none, from=2-1, to=1]
	\arrow["\sqsupseteq"{description}, draw=none, from=2-7, to=0]
\end{tikzcd}\]
The right-adjoints can be computed explicitely as
\begin{equation}\label{eq:RTS}
(R\blacktriangleright T) (y,z)=\bigsqcap_{x\in X}R(x,y)\rhd T(x,z)
    \quad\quad
(T\blacktriangleleft S)(x,y)=\bigsqcap_{z\in Z} T(x,z) \lhd S(y,z)
\end{equation}
From this remark and the quantale laws in \ref{sec:quantale-laws} the following two equations are easy to derive.
\end{remark}

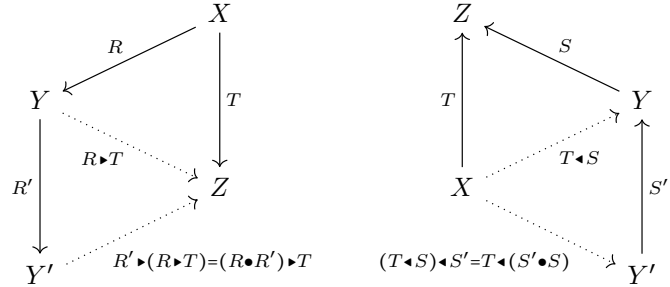
\begin{lemma}\label{lem:double-blacktriangle}
% https://q.uiver.app/#q=WzAsOCxbNSwwLCJaIl0sWzUsMiwiWCJdLFs3LDEsIlkiXSxbNywzLCJZJyJdLFsyLDAsIlgiXSxbMiwyLCJaIl0sWzAsMSwiWSJdLFswLDMsIlknIl0sWzEsMCwiVCJdLFsyLDAsIlMiLDJdLFszLDIsIlMnIiwyXSxbMSwyLCJUXFxibGFja3RyaWFuZ2xlbGVmdCBTIiwyLHsic3R5bGUiOnsiYm9keSI6eyJuYW1lIjoiZG90dGVkIn19fV0sWzEsMywiKFRcXGJsYWNrdHJpYW5nbGVsZWZ0IFMpXFxibGFja3RyaWFuZ2xlbGVmdCBTJz1UXFxibGFja3RyaWFuZ2xlbGVmdCAoUydcXGJ1bGxldCBTKSIsMix7InN0eWxlIjp7ImJvZHkiOnsibmFtZSI6ImRvdHRlZCJ9fX1dLFs0LDUsIlQiXSxbNCw2LCJSIiwyXSxbNiw3LCJSIiwyXSxbNiw1LCJUXFxibGFja3RyaWFuZ2xlcmlnaHQgUlxcYmxhY2t0cmlhbmdsZXJpZ2h0IFQiLDIseyJzdHlsZSI6eyJib2R5Ijp7Im5hbWUiOiJkb3R0ZWQifX19XSxbNyw1LCJSJ1xcYmxhY2t0cmlhbmdsZXJpZ2h0KFJcXGJsYWNrdHJpYW5nbGVyaWdodCBUKT0oUlxcYnVsbGV0IFInKVxcYmxhY2t0cmlhbmdsZXJpZ2h0IFQiLDIseyJzdHlsZSI6eyJib2R5Ijp7Im5hbWUiOiJkb3R0ZWQifX19XV0=
\[\begin{tikzcd}
	&& X &&& Z \\
	Y &&&&&&& Y \\
	&& Z &&& X \\
	{Y'} &&&&&&& {Y'}
	\arrow["T", from=3-6, to=1-6]
	\arrow["S"', from=2-8, to=1-6]
	\arrow["{S'}"', from=4-8, to=2-8]
	\arrow["{T\blacktriangleleft S}"', dotted, from=3-6, to=2-8]
	\arrow["{(T\blacktriangleleft S)\blacktriangleleft S'=T\blacktriangleleft (S'\bullet S)}"', dotted, from=3-6, to=4-8, pos=0.7]
	\arrow["T", from=1-3, to=3-3]
	\arrow["R"', from=1-3, to=2-1]
	\arrow["{R'}"', from=2-1, to=4-1]
	\arrow["{R\blacktriangleright T}"', dotted, from=2-1, to=3-3]
	\arrow["{R'\blacktriangleright(R\blacktriangleright T)=(R\bullet R')\blacktriangleright T}"', dotted, from=4-1, to=3-3, pos=0.3]
\end{tikzcd}\]
\end{lemma}

\subsection{Duality} 
In sets, every relation $R:A\looparrowright B$ has a converse $R^o:B\looparrowright A$. Enriching over a commutative $\Omega$ the opposite relation is not a converse anymore since it is now of type $R^o:B^o\looparrowright A^o$. Enriching over a non-commutative quantale, the opposite relation is in a different category since it is now enriched over $\Omega^o$.

\subsection{Weighted upsets and downsets} A \emph{weighted downset} of $X$, or a presheaf on $X$, is a quantale relation $X\looparrowright 1$ where $1$ is the one-element $\Omega$-space. A \emph{weighted upset} of $Y$, or a co-presheaf on $Y$, is a quantale relation $1\looparrowright Y$. In particular, for a presheaf $\phi$ on $X$, we have
$$X(x',x)\cdot \phi(x) \sqsubseteq \phi(x')
\quad\quad\quad
X(x',x)\sqsubseteq\phi(x')\lhd \phi(x)$$
and for a co-presheaf $\psi$ on $Y$ we have
$$
\psi(y)\cdot Y(y,y')\sqsubseteq\psi(y')
\quad\quad\quad
Y(y,y')\sqsubseteq \psi(y)\rhd\psi(y')
$$
where we simplify the notation $\phi(x,y)$ and $\psi(x,y)$ by dropping the variable of type $1$.

\begin{definition}\label{def:DU}
    Define the $\Omega$-space $\mathcal DX=\Rel(\Omega)(X,1)$ of weighted downsets via
    $$\mathcal DX(\phi,\phi')=\phi\blacktriangleright\phi'$$
    and the $\Omega^o$-space $\mathcal UA=\Rel(\Omega)(1,A)$ of weighted upsets via 
    $$\mathcal UA(\psi,\psi')=\psi \blacktriangleleft \psi'$$
\end{definition}

\begin{remark}\label{rmk:DU} Recall  that \eqref{eq:RTS} implies 
\begin{align*}
\phi\blacktriangleright\phi' & =\bigsqcap_{x\in X} (\phi x\rhd \phi' x)\\
\psi \blacktriangleleft \psi' &=\bigsqcap_{a\in A} (\psi a\lhd \psi' a)
\end{align*}
This ensures that the Yoneda embeddings $X\to\mathcal DX, x\mapsto X(-,x)$ and $A\to\mathcal UA, a\mapsto A(a,-)$ are functorial.
\end{remark}

\begin{lemma}\label{lem:blackrightleft}
Let $\cont:X\looparrowright A$, $\phi\in\mathcal DX$, $\psi\in\mathcal UA$. Then
\begin{equation}\label{eq:adjunction}
(\phi\blacktriangleright \cont)\blacktriangleleft \psi = \phi\blacktriangleright (\cont \blacktriangleleft \psi),
\end{equation}
or, in a diagram:
% https://q.uiver.app/#q=WzAsNCxbMCwwLCJYIl0sWzIsMCwiQSJdLFswLDIsIjEiXSxbMiwyLCIxIl0sWzAsMSwiXFxiZiBJIl0sWzAsMiwiXFxwaGkiLDJdLFszLDEsIlxccHNpIiwyXSxbMiwzLCJcXHBoaVxcYmxhY2t0cmlhbmdsZXJpZ2h0IFxcYmYgSSBcXGJsYWNrdHJpYW5nbGVsZWZ0IFxccHNpIiwyXV0=
\[\begin{tikzcd}
	X && A \\
	\\
	1 && 1
	\arrow["{\bf I}", from=1-1, to=1-3]
	\arrow["\phi"', from=1-1, to=3-1]
	\arrow["\psi"', from=3-3, to=1-3]
	\arrow["{\phi\blacktriangleright \bf I \blacktriangleleft \psi}"', from=3-1, to=3-3]
\end{tikzcd}\]
\end{lemma}

\begin{proof}
This is an instance of an equation presented in Remark \ref{rmk:blacktriangle}.
\end{proof}

\begin{remark}
    $(\mathcal D A^o)^o=\mathcal U A$. 
\end{remark}

\subsection{Duality}
We have
$$ \Rel(\Omega) \cong \Rel(\Omega^o)^o
\quad\quad\quad
\Spa(\Omega) \cong \Spa(\Omega^o)^c
$$
where $(-)^o$ reverses 1-cells and $(-)^c$ reverses 2-cells.
The iso on the left maps $R:A\looparrowright B$ to $R^o:B^o\to A^o$. The iso on the right maps $F:A\to B$ to $F^o:A^o\to B^o$.
Moreover, there
% \begin{proposition}
% There 
are isomorphisms of posets, natural in $X$ and $A$,
$$\Rel(\Omega)(X,A) \cong \Spa(\Omega)^o(\mathcal DX, A)\cong \Spa(\Omega)^c(X,\mathcal UA).$$
%\end{remark}

\subsection{Weighted (Co)Limits}\label{sec:weighted-colimits}

We follow Stubbe~\cite{Stubbe05a}. 

% https://q.uiver.app/#q=WzAsMTIsWzQsMiwiQiJdLFsyLDEsIkMiXSxbNCwwLCJEIl0sWzgsMiwiQiJdLFs4LDAsIkQiXSxbNiwxLCJDIl0sWzIsMywiQSJdLFs0LDQsIkIiXSxbNiwzLCJBIl0sWzgsNCwiQiJdLFswLDAsIkQiXSxbMCwyLCJCIl0sWzEsMCwiXFxwaGlcXGJsYWNrdHJpYW5nbGVyaWdodCBCKEcsLSkiLDIseyJsYWJlbF9wb3NpdGlvbiI6NDAsInN0eWxlIjp7ImJvZHkiOnsibmFtZSI6ImJhcnJlZCJ9fX1dLFsyLDAsIkIoRywtKSIsMCx7InN0eWxlIjp7ImJvZHkiOnsibmFtZSI6ImJhcnJlZCJ9fX1dLFsyLDEsIlxccGhpIiwyLHsic3R5bGUiOnsiYm9keSI6eyJuYW1lIjoiYmFycmVkIn19fV0sWzMsNCwiQigtLEcpIiwyLHsic3R5bGUiOnsiYm9keSI6eyJuYW1lIjoiYmFycmVkIn19fV0sWzMsNSwiQigtLEcpIFxcYmxhY2t0cmlhbmdsZWxlZnQgXFxwc2kiLDAseyJsYWJlbF9wb3NpdGlvbiI6NjAsInN0eWxlIjp7ImJvZHkiOnsibmFtZSI6ImJhcnJlZCJ9fX1dLFs1LDQsIlxccHNpIiwwLHsic3R5bGUiOnsiYm9keSI6eyJuYW1lIjoiYmFycmVkIn19fV0sWzYsNywie1xccm0gY29saW19X1xccGhpXFwsIEciXSxbOCw5LCJcXGxpbV9cXHBzaSBHIl0sWzEwLDExLCJHIl0sWzUsMTUsIlxcbGUiLDEseyJzaG9ydGVuIjp7InRhcmdldCI6MjB9LCJzdHlsZSI6eyJib2R5Ijp7Im5hbWUiOiJub25lIn0sImhlYWQiOnsibmFtZSI6Im5vbmUifX19XSxbMSwxMywiXFxsZSIsMSx7InNob3J0ZW4iOnsic291cmNlIjoyMCwidGFyZ2V0IjoyMH0sInN0eWxlIjp7ImJvZHkiOnsibmFtZSI6Im5vbmUifSwiaGVhZCI6eyJuYW1lIjoibm9uZSJ9fX1dXQ==
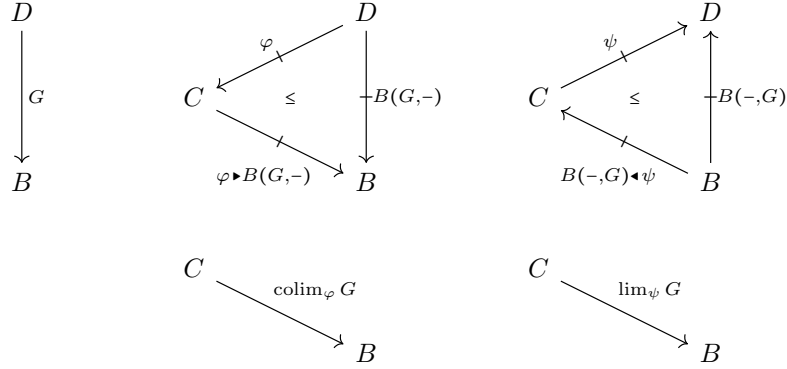
\[\begin{tikzcd}
	D &&&& D &&&& D \\
	&& C &&&& C \\
	B &&&& B &&&& B \\
	&& C &&&& C \\
	&&&& B &&&& B
	\arrow["{\phi\blacktriangleright B(G,-)}"'{pos=0.8}, "\shortmid"{marking}, from=2-3, to=3-5]
	\arrow[""{name=0, anchor=center, inner sep=0}, "{B(G,-)}", "\shortmid"{marking}, from=1-5, to=3-5]
	\arrow["\phi"', "\shortmid"{marking}, from=1-5, to=2-3]
	\arrow[""{name=1, anchor=center, inner sep=0}, "{B(-,G)}"', "\shortmid"{marking}, from=3-9, to=1-9]
	\arrow["{B(-,G) \blacktriangleleft \psi}"{pos=0.2}, "\shortmid"{marking}, from=3-9, to=2-7]
	\arrow["\psi", "\shortmid"{marking}, from=2-7, to=1-9]
	\arrow["{{\rm colim}_\phi\, G}", from=4-3, to=5-5,{pos=0.4}]
	\arrow["{\lim_\psi G}", from=4-7, to=5-9,{pos=0.4}]
	\arrow["G", from=1-1, to=3-1]
	\arrow["\le"{description}, draw=none, from=2-7, to=1]
	\arrow["\le"{description}, draw=none, from=2-3, to=0]
\end{tikzcd}\]
The \emph{weighted colimit} ${\rm colim}_\phi\,G$ is the (unique, if it exist) solution of
\begin{gather*}
B({\rm colim}_\phi G,b) = \phi\blacktriangleright B(G,b)
\end{gather*}
The \emph{weighted limit} ${\rm lim}_\psi G$ is the (unique, if it exists) solution of
\begin{gather*}
B(b, {\rm lim}_\psi G) = B(b,G)\blacktriangleleft\psi
\end{gather*}
We will be mostly concerned with the case $C=1$ where $\phi$ is a weighted downset of $D$ and $\psi$ is a weighted upset of $D$.

\begin{example} \label{exle:weighted-limit} We continue from Example~\ref{exle:omega}.
\begin{enumerate}
\item In $\Spa(2)$, for $D=B$ and $G$ the identity, ${\rm colim_\phi}G$ is the join of $\phi$ and ${\rm lim}_\psi$ is the meet of $\psi$.

\item Weighted (co)limits in (ultra)metric spaces were studied by Wagner~\cite{Wagner:PhD,Wagner97} and Rutten et.al.~\cite{Rutten96,Rutten98,Bonsangue98}.

\item Let $A$ be $\mathcal P(\Sigma^\ast)$-space, that is, a generalized non-deterministic automaton. 

Let $f$ be a function from the underlying set $|A|$ of $A$ to $\{\emptyset,\epsilon\}$, interpreted as the characteristic function of the set of final states of $A$. Then $f$ can be extended to the ``observability" presheaf $\phi\in \mathcal DA$ so that $\phi(q)$ is the language accepted by $A$ in state $q$. Technically, $\phi$ is the colimit of $|A|\to A \to \mathcal DA$ weighted by $f$. 

Dually, let $i$ be a function from the underlying set $|A|$ of $A$ to $\{\emptyset,\epsilon\}$, interpreted as the characterstic function of the set of initital states of $A$. Then $i$ can be extended to the ``reachability" presheaf $\psi\in\mathcal UA$ so that $\psi(q)$ is the language of words leading from an initial state to $q$. Technically, $\psi$ is the colimit of $|A|\to A \to \mathcal UA$ weighted by $i$. 
\end{enumerate}
To be continued in Example~\ref{exle:tensor}.
\end{example}

\subsection{Join, Meet, Tensor, Power}

The so-called conical (co)limits are the special case where the (co)presheaf has constant value $e\in\Omega$. In that case we can drop the (co)presheaf from the notation and have
$$
{\rm colim}\, G=\bigsqcup_{d\in D} Gd
\quad\quad\quad
{\rm lim}\, G=\bigsqcap_{d\in D} Gd
$$
The tensor (copower) and the power (cotensor) are the special case where $D=1$. In this case, writing $g$ for the value of $G$ and $r$ for the value of $\phi$ (resp $\psi$) in $\Omega$, we have 
$$
B(g\star r,b)=r \rhd B(g,b) \quad\quad B(b,g \uparrow r)= B(b,g)\lhd r
$$
where we write $- \star r$ for tensoring with $r\in\Omega$ and $-\uparrow r$ for taking to the power of $r$. 

\begin{remark}
In case of $B=\Omega$ enriched over itself, we have $\star=\cdot$ and $\uparrow=\lhd$.
\end{remark}
\begin{example} \label{exle:tensor} We continue from Example~\ref{exle:weighted-limit}.
\begin{enumerate}
\item In $[0,\infty]$, seen as enriched over itself, tensoring with $r\in[0,\infty]$ is addition $+r$ and power is truncated subtraction $\dotminus r$. In the LMS $\mathbb R$ with its natural order tensoring is truncated subtraction and power is addition.

\item For a determinstic automata $A$, considered as an element of $\Spa(\mathcal P(\Sigma^\ast))$, tensoring a state with $\{a\}$ gives the $a$-successor, that is, $q\cdot \{a\}$ is the $a$-successor of $q$ in $A$.

In the history automaton $\mathcal P(\Sigma^\ast)$, which has  $\rhd$ as hom, tensoring a state $M$ with $L$ is given by the product $M\cdot L$ of languages.

In the prophecy automaton $\mathcal P(\Sigma^\ast)$, which has $\lhd$ as hom, tensoring a state $M$ with $L$ is given by the left-residual $L\rhd M$ of languages. In particular, tensoring with $\{w\}$ is the so-called Brzozowski derivative.  

In each of these examples, tensoring formalizes a notion of successor.
\end{enumerate}
\end{example}

\begin{proposition} Limits and colimits in an $\Omega$-space that has tensors and powers as well as all conical limits can be computed explicitely as 
\begin{align*}
{\rm colim}_\phi G &= \bigsqcup _{d\in D}(Gd\star\phi d) \\
{\rm lim}_\psi G &= \bigsqcap_{d\in D} (Gd\uparrow\psi d).
\end{align*}
\end{proposition}

\begin{proof} The proposition is well-known. We present the proof because similar computations will play a role later on.
\begin{align*}
B({\rm colim}_\phi G,b) &= \phi\blacktriangleright B(G,b)\\
&= \bigsqcap_{d\in D}(\phi x\rhd B(Gx,b))\\
&= \bigsqcap_{d\in D} B(Gx\star\phi x,b)\\
&= B(\bigsqcup_{d\in D} Gx\star\phi x,b).
\end{align*} 
\begin{align*}
B(b,{\rm lim}_\psi G) &= B(b,G)\blacktriangleleft \psi\\
& = \bigsqcap_{d\in D} (B(b,Gd)\lhd\psi d)\\
& = \bigsqcap_{d\in D} B(b,Gx\uparrow\psi x)\\
& = B(b,\bigsqcap_{d\in D} (Gx\uparrow\psi x)).
\end{align*}
\end{proof}

\begin{corollary}
The composition $1\stackrel \psi\looparrowright X \stackrel\phi \looparrowright 1$ is the colimit of $\psi:X\to\Omega$ weighted by $\phi$
\begin{equation}\label{eq:psi-phi}
{\rm colim_\phi}\psi = \bigsqcup_x \psi x\cdot \phi x = \psi\bullet \phi.
\end{equation}
\end{corollary}

The following proofs exemplify how the introduced  notation enables straightforward algebra-style proofs.

\subsection{Yoneda preserves limits} \label{sec:Yoneda-preserves-limits} The Yoneda embedding $C\to\mathcal DC$ preserves limits and the Yoneda embedding $C\to\mathcal UC$ preserves colimits.

\begin{proof} %We write the proof out in detail as a test for the blacktriangle calculus. 
Let $G:X\to C$ and $g\in\mathcal UX$ and $j\in\mathcal DX$. Consider the Yoneda embedding $Y:C\to\mathcal DC$, $c\mapsto C(-,c)$. We have to show 
$$Y({\rm lim}_g\, G) = {\rm lim}_g\, YG.$$ 
Observe that, by the definition of limits in $C$, 
$$Y({\rm lim}_g\,G) = YG\blacktriangleleft g$$
and use Lemma~\ref{lem:blackrightleft} with $\cont(c,x) = \mathcal C(c,Gx)$:
\begin{align*}
\mathcal DC(\phi,Y ({\rm lim}_g\, G))
&= \phi\blacktriangleright Y ({\rm lim}_g\, G)
& \text{def of } \mathcal DC
\\
&= \phi\blacktriangleright (YG \blacktriangleleft g)
& \text{def of lim in } C
\\
&= (\phi\blacktriangleright YG) \blacktriangleleft g
& \eqref{eq:adjunction}
\\
&= \mathcal DC(\phi,YG)\blacktriangleleft g
& \text{def of } \mathcal DC
\\
&= \mathcal DC(\phi,{\rm lim}_g\,YG)
& \text{def of lim in } \mathcal DC 
\end{align*}
\end{proof}

\begin{lemma}\label{lem:lim-colim} For $\phi\in\mathcal DX$, $\psi\in\mathcal UA$, $\cont:X\looparrowright A$, hence $\cont:X\to \mathcal UA$ and $\cont:A\to\mathcal DX$, we have
$$\mathcal DX(\phi,{\rm lim}_\psi\, \cont)=\mathcal UA({\rm colim}_\phi\,\cont,\psi)$$
\end{lemma}

\begin{proof}
\begin{align*}
\mathcal DX(\phi,{\rm lim}_\psi\, \cont)
&= \mathcal DX(\phi,\cont)\blacktriangleleft \psi
& \text{def of lim}
\\
&= (\phi\blacktriangleright \cont)\blacktriangleleft \psi
& \text{def of } \mathcal DX
\\
&= \phi\blacktriangleright (\cont\blacktriangleleft \psi)
& \eqref{eq:adjunction}
\\
&= \phi\blacktriangleright \mathcal UA(\cont, \psi)
& \text{def of } \mathcal UA
\\
&= \mathcal UA({\rm colim}_\phi\,\cont, \psi)
& \text{ def of colim}
\end{align*}
\end{proof}

\section{MacNeille Completion}\label{sec:macneille}

Valuable resources on the MacNeille completion of an $\Omega$-enriched category are Stubbe~\cite{Stubbe05a}, Garner~\cite{Garner:topological}, Shen and Zhang~\cite{Shen-Zhang} and Fujii~\cite{Fujii}. We review the results we need later and set up the notation. 

We define the MacNeille completion for a relation rather than only for a category. This will allow us to define a complete and cocomplete category enriched over $\Omega$ simply by specifying a set-theoretic function $X\times A\to\Omega$. This is convenient both for future applications to logic  as well as for defining canonical extensions. 

We use the language of formal concept analysis:  A ``context'' $\cont:X\looparrowright A$ relates ``objects'' in $X$ and ``attributes'' in $A$;  a ``concept'' is a pair consisting of a weighted downset of objects (called the ``extent'') and a weighted upset of attributes (called the ``intent'').

%\begin{definition}
\subsection{The MacNeille Adjunction}
Recall the blacktriangle notation from Remark~\ref{rmk:blacktriangle}, Definition~\ref{def:DU} of weighted up- and downsets as well as the quantale laws from Section~\ref{sec:quantale-laws}.

The MacNeille adjunction induced by a weighted relation $\cont:X\looparrowright A$ is given by $\cont^\uparrow(\phi)=\phi\blacktriangleright \cont$ and by $\cont^\downarrow(\psi)=\cont\blacktriangleleft\psi$.
% https://q.uiver.app/#q=WzAsMixbMCwwLCJcXG1hdGhjYWwgRCBYIl0sWzIsMCwiXFxtYXRoY2FsIFUgQSJdLFswLDEsInBeXFx1cGFycm93IiwwLHsiY3VydmUiOi0yfV0sWzEsMCwicF5cXGRvd25hcnJvdyIsMCx7ImN1cnZlIjotMn1dXQ==
\[
\begin{tikzcd}
	{\mathcal{D}X} && {\mathcal{U}A}
	\arrow["{\cont^\uparrow \;=\; -\blacktriangleright \cont}", curve={height=-12pt}, from=1-1, to=1-3]
	\arrow["{\cont^\downarrow \;=\; \cont\blacktriangleleft -}", curve={height=-12pt}, from=1-3, to=1-1]
	\arrow[phantom, "\bot", from=1-1, to=1-3]
\end{tikzcd}
\]
%\end{definition}
\begin{remark}
From Remark~\ref{rmk:DU} it follows that 
\begin{align*}
(\phi\blacktriangleright \cont)(a)
=\phi\blacktriangleright\cont(-,a) 
=\, \bigsqcap_x\ \phi(x)\rhd\cont(x,a)
\\
(\cont\blacktriangleleft\psi)(x)
=\cont(x,-)\blacktriangleleft\psi
=\,\bigsqcap_a\ \cont(x,a)\lhd\psi(a)
\end{align*}
\end{remark}

%\begin{ak}
%Georgescu and Popescu~\cite{georgescu2003non} investigate non-commutative fuzzy Galois connections ... add one more sentence here or somewhere else ...
%\end{ak} 

\begin{lemma}\label{lem:adjunction}
    $\cont^\uparrow \dashv \cont^\downarrow$.
\end{lemma}

\begin{proof}
    We have to show $\mathcal UA(\phi\blacktriangleright \cont,\psi)=\mathcal DX(\phi,\cont\blacktriangleleft\psi)$, which is equivalent to 
    $$
    (\phi\blacktriangleright \cont)\blacktriangleleft \psi = \phi\blacktriangleright (\cont \blacktriangleleft \psi).
    $$
    and was proved in \eqref{eq:adjunction}.
\end{proof}
\begin{remark}
Since the MacNeille adjunction lives in an order-enriched category, $\Spa(\Omega)$, the induced monad and comonad are \href{https://ncatlab.org/nlab/show/idempotent+monad#definition}{idempotent}. This in turn implies that the monad $\cont^\downarrow \cont^\uparrow$ is a closure operator on $\mathcal DX$ and the comonad $\cont^\uparrow \cont^\downarrow$  on $\mathcal UA$ is an interior operator. Since $\mathcal UA=\mathcal (DA^o)^o$, one can think of $\cont^\uparrow$ and $\cont^\downarrow$ also as contravariant functors and of both  $\cont^\downarrow \cont^\uparrow$ and  $\cont^\uparrow \cont^\downarrow$ as monads (closure operators wrt the ``inclusion'' order). 
\end{remark}

\subsection{Idempotence}\label{sec:idempotence} That the monads induced by the adjunction are \href{https://ncatlab.org/nlab/show/idempotent+monad#definition}{idempotent} implies
\begin{gather*}
\cont^\uparrow \cont^\downarrow \cont^\uparrow = \cont^\uparrow\\
\cont^\downarrow \cont^\uparrow \cont^\downarrow = \cont^\downarrow\\
\phi=\cont\blacktriangleleft \psi \ \Leftrightarrow \ \phi\blacktriangleright \cont = \psi
\end{gather*}

In particular, the following two equations will be useful.

\begin{lemma}
\begin{gather*}(\cont\blacktriangleleft \cont(x,-)) \blacktriangleright \cont \ =\ \cont(x,-)\\
\cont\blacktriangleleft (\cont(-,a) \blacktriangleright \cont)\ =\ \cont(-,a)
\end{gather*}
\end{lemma}

We make use of idempotence to define the MacNeille completion of a relation $\cont$ as the space of fixed points of the induced closure operators.

%\begin{definition}
\subsection{MacNeille Completion}\label{sec:macneille-completion}
Let $\cont:X\looparrowright A$ be a weighted relation. The MacNeille completion $\mathcal M(\cont)$ of $\cont$ has as objects pairs $\kappa=(\phi,\psi)=(\val{\kappa},\descr{\kappa})$ (often referred to as {\em concepts}) such that $\val{\kappa}\blacktriangleright \cont=\descr{\kappa}$ and $\val{\kappa}= \cont \blacktriangleleft\descr{\kappa}$ and homs 
\[
\mathcal M(\cont)(\kappa,\kappa')= \val{\kappa}\blacktriangleright \val{\kappa'}= \descr{\kappa}\blacktriangleleft \descr{\kappa'}.
\]
The MacNeille completion of an $\Omega$-space $C$ is the MacNeille completion of its hom $C(-,-):C\looparrowright C$. 

Since the components of a pair $(\phi,\psi)\in\mathcal M(\cont)$ determine each other, we may identify a pair by any of its two components.
%\end{definition}

\begin{remark}
Lemma~\ref{lem:adjunction} implies that the homs are well-defined:
\begin{align*}
\phi\blacktriangleright \phi' 
&= \phi\blacktriangleright(\cont\blacktriangleleft \psi')\\
&= (\phi\blacktriangleright \cont)\blacktriangleleft \psi'\\
&= \psi\blacktriangleleft\psi'
\end{align*}
\end{remark}

%\begin{remark}\label{rmk:notation} For future calculations we will need
%\begin{gather*}
%(\phi\blacktriangleright \cont) (a) 
%= (\bigsqcap_x \phi x \rhd \cont(x,-))(a)
%= \bigsqcap_x \phi x \rhd \cont(x,a)
%= \phi \blacktriangleright \cont(-,a)
%\\
%(\cont\blacktriangleleft \psi) (x) 
%= (\bigsqcap_a \cont(-,a)\lhd \psi a)(x)
%= \bigsqcap_a \cont(x,a)\lhd \psi a
%= \cont(x,-)\blacktriangleleft \psi 
%\end{gather*}
%\end{remark}

\subsection{The MacNeille Embedding}
We will see below that the MacNeille completion $\mathcal M(\cont)$ of a relation $\cont:X\looparrowright A$ is indeed a completion of the image of $X+A$ in $\mathcal M(\cont)$. To make this precise we define
\begin{align*}
    \overline{(-)} : X+A & \longrightarrow \mathcal M(\cont)\\
    \overline x  & \ = \ X(-,x)\blacktriangleright \cont  \\
    \overline a  & \ = \ \cont \blacktriangleleft A(a,-)
\end{align*}
We denote the image of $\overline{(-)}$ in $\mathcal M(\cont)$ by $\overline\cont$:
$$
X+A \longrightarrow \overline\cont \longrightarrow \mathcal M(\cont)
$$
The factorisation we have in mind here is the one that has $\overline \cont \longrightarrow \mathcal M(\cont)$ as fully faithful. Since we think of a context $\cont$ as a specification of $\overline \cont$ we sometimes call $\overline{(-)}$ the MacNeille embedding even if, in general, only $\overline \cont \longrightarrow \mathcal M(\cont)$ is fully faithful.

\begin{remark}\label{rmk:yoneda}
Note that, due to the Yoneda lemma, we have
\begin{align*}
    \overline x  & \ = \ X(-,x)\blacktriangleright \cont =\cont(x,-) \in\mathcal UA \\
    \overline a  & \ = \ \cont \blacktriangleleft A(a,-) = \cont(-,a) \in\mathcal DX
\end{align*}
\end{remark}

\begin{lemma}\label{lem:mc-distance}
\begin{align*}
\mathcal M(\cont)(\overline x,\overline a) &= \cont(x,a)\\
\mathcal M(\cont)(\overline a,\overline a') &= \cont(-,a)\blacktriangleright \cont(-,a')\\
\mathcal M(\cont)(\overline x,\overline x') &= \cont(x,-)\blacktriangleleft \cont(x',-)\\
\mathcal M(\cont)(\overline a,\overline x) &= \cont(-,a) \blacktriangleright \cont\blacktriangleleft \cont(x,-)
\end{align*}   
\end{lemma}

\begin{proof} For item 1, we compute
\begin{align*}
\mathcal M(\cont)(\overline x,\overline a) 
&= (\cont\blacktriangleleft \cont(x,-)) \blacktriangleright \cont(-,a)
& \text{definitions}
\\&= (\cont\blacktriangleleft \cont(x,-)) \blacktriangleright (\cont\blacktriangleleft A(-,a))
& \text{Yoneda \ref{rmk:yoneda}}
\\&= ((\cont\blacktriangleleft \cont(x,-)) \blacktriangleright \cont)\blacktriangleleft A(-,a)
& \eqref{eq:adjunction}
\\&= \cont(x,-) \blacktriangleleft A(-,a)
& \text{Idempotence \ref{sec:idempotence}}
\\&= \cont(x,a)
& \text{Yoneda \ref{rmk:yoneda}}
\end{align*}
Items 2 to 4 follow immediately from the respective definitions and Remark~\ref{rmk:yoneda}.
\end{proof}

The importance of Lemma~\ref{lem:mc-distance} lies in the fact that while $\cont:X\looparrowright A$ seemingly only specifies a distance  from $x$ to $a$ for each pair $(x,a)$, it in fact determines all distances on $X+A$. %\begin{ak}$X+A$ as a category\end{ak} 
A corollary is that the restriction $|\cont|$ of $\cont$ to the discrete spaces $|X|$ and $|A|$ defines the same MacNeille completion. Indeed, due to $|\cont|(x,a)=\cont(x,a)$ we have

\begin{proposition} \ 
\begin{enumerate}
\item Let $\phi\in\mathcal D|X|$ and $\psi\in\mathcal U|A|$. Then $\phi\blacktriangleright |\cont| \in \mathcal UA$ and $\psi\blacktriangleleft |\cont| \in \mathcal DX$.
\item $\mathcal M(|\cont|)=\mathcal M(\cont)$.
\end{enumerate}
\end{proposition}

This simplifies the specification of $\Omega$-spaces via contexts as we can now start with relations and predicates that are not necessarily monotone.  

% https://q.uiver.app/#q=WzAsNSxbMCwwLCJcXG1hdGhjYWwgRFgiXSxbMiwyLCJcXG1hdGhjYWwgTShJKSJdLFs0LDAsIlxcbWF0aGNhbCBVQSJdLFswLDQsIlgiXSxbNCw0LCJBIl0sWzMsMCwieFxcbWFwc3RvIFgoLSx4KSJdLFs0LDIsImFcXG1hcHN0byBBKGEsLSkiLDJdLFszLDQsIlxccm0gSSIsMix7InN0eWxlIjp7ImJvZHkiOnsibmFtZSI6ImRvdHRlZCJ9fX1dLFszLDEsInhcXG1hcHN0byBcXG92ZXJsaW5lIHgiLDFdLFs0LDEsImFcXG1hcHN0byBcXG92ZXJsaW5lIGEiLDFdLFswLDIsIi1cXGJsYWNrdHJpYW5nbGVyaWdodCBcXHJtIEkiLDAseyJjdXJ2ZSI6LTJ9XSxbMiwwLCJ7XFxybSBJfVxcYmxhY2t0cmlhbmdsZWxlZnQgLSIsMCx7ImN1cnZlIjotMn1dLFsxLDAsIiIsMCx7InN0eWxlIjp7InRhaWwiOnsibmFtZSI6Imhvb2siLCJzaWRlIjoiYm90dG9tIn19fV0sWzEsMiwiIiwwLHsic3R5bGUiOnsidGFpbCI6eyJuYW1lIjoiaG9vayIsInNpZGUiOiJ0b3AifX19XSxbMCwxLCIiLDAseyJjdXJ2ZSI6Miwic3R5bGUiOnsiYm9keSI6eyJuYW1lIjoiZGFzaGVkIn19fV0sWzAsMiwiXFxib3QiLDEseyJzdHlsZSI6eyJib2R5Ijp7Im5hbWUiOiJub25lIn0sImhlYWQiOnsibmFtZSI6Im5vbmUifX19XSxbMCwxLCJcXGRhc2h2IiwxLHsiY3VydmUiOjEsInN0eWxlIjp7ImJvZHkiOnsibmFtZSI6Im5vbmUifSwiaGVhZCI6eyJuYW1lIjoibm9uZSJ9fX1dLFsyLDEsIiIsMSx7ImN1cnZlIjotMiwic3R5bGUiOnsiYm9keSI6eyJuYW1lIjoiZGFzaGVkIn19fV0sWzEsMiwiXFxkYXNodiIsMSx7ImN1cnZlIjoxLCJzdHlsZSI6eyJib2R5Ijp7Im5hbWUiOiJub25lIn0sImhlYWQiOnsibmFtZSI6Im5vbmUifX19XSxbNSwzXSxbNSw0XSxbNSwxLCJbLV0iLDIseyJjdXJ2ZSI6MX1dXQ==
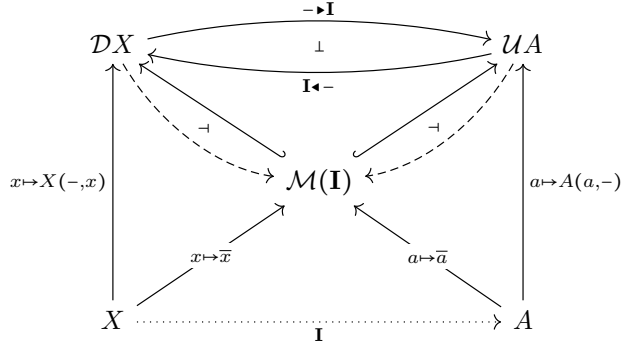
\[\begin{tikzcd}
	{\mathcal DX} &&&& {\mathcal UA} \\
	\\
	&& {\mathcal M(\cont)} \\
	\\
	X &&&& A
	\arrow["{x\mapsto X(-,x)}", from=5-1, to=1-1]
	\arrow["{a\mapsto A(a,-)}"', from=5-5, to=1-5]
	\arrow["{\cont}"', dotted, from=5-1, to=5-5]
	\arrow["{x\mapsto \overline x}"{description}, from=5-1, to=3-3]
	\arrow["{a\mapsto \overline a}"{description}, from=5-5, to=3-3]
	\arrow["{-\blacktriangleright \cont}", curve={height=-12pt}, from=1-1, to=1-5]
	\arrow["{\cont\blacktriangleleft -}", curve={height=-12pt}, from=1-5, to=1-1]
	\arrow[hook', from=3-3, to=1-1]
	\arrow[hook, from=3-3, to=1-5]
	\arrow[curve={height=16pt}, dashed, from=1-1, to=3-3]
	\arrow["\bot"{description}, draw=none, from=1-1, to=1-5]
	\arrow["\dashv"{description}, curve={height=8pt}, draw=none, from=1-1, to=3-3]
	\arrow[curve={height=-16pt}, dashed, from=1-5, to=3-3]
	\arrow["\dashv"{description}, curve={height=8pt}, draw=none, from=3-3, to=1-5]
\end{tikzcd}\]
Before we discuss the dashed arrows, we note that in general the Yoneda embeddings (vertical arrows) do not factor through $\mathcal M(\cont)$, unless we are in the special case of the next proposition.\footnote{But note that $X(-,x)\blacktriangleright\cont$ and $\cont\blacktriangleleft A(a,-)$ are in $\mathcal M(\cont)$.}

\begin{proposition}
If $\cont$ is the internal hom of a quantale space $C$, then $\overline\cont$ and $C$ are isomorphic quantale spaces. In particular, the Yoneda embeddings factor through $\mathcal M(\cont)$.
\end{proposition}

\begin{proof}
Let $\cont$ be the internal hom of an $\Omega$-space $C$. Let $x,a\in C$. Then 
\begin{align*}
\mathcal M(\cont)(\overline x,\overline a) 
&= (C(-,-)\blacktriangleleft C(x,-)) \blacktriangleright C(-,a)
& \text{definitions}
\\
&= C(-,x) \blacktriangleright C(-,a)
& \text{Yoneda}
\\
&= C(x,a) 
& \text{Yoneda}
\end{align*}
\end{proof}

\begin{proposition}
    Let $\cont:X\looparrowright A$ be an $\Omega$-relation. $\mathcal M(\cont)$ is both a full reflective subcategory of  $\mathcal DX$ and a full coreflective subcategory of $\mathcal UA$. In particular, $\mathcal M(\cont)$ is complete and cocomplete. 
\end{proposition}

\begin{proof} The argument is well-known and we only sketch it. 
    $\mathcal M(\cont)$ is the category of \href{https://ncatlab.org/nlab/show/idempotent+monad#AlgebrasForAnIdempotentMonad}{algebras for the monad} $\cont^\downarrow\cont^\uparrow$ and the category of coalgebras for the comonad $\cont^\uparrow\cont^\downarrow$, hence \href{https://ncatlab.org/nlab/show/reflective+subcategory#definition}{reflective} in $\mathcal DX$ and coreflective in $\mathcal UA$.
\end{proof}

\begin{corollary}\label{cor: limits in mc}
\begin{enumerate}
\item The embedding $\overline \cont\to\mathcal M(\cont)$ preserves limits and colimits.
\item
The colimit of $X\to\mathcal M(\cont)$ weighted by $\phi\in\mathcal DX$ is $\phi\blacktriangleright\cont$. 
\item The limit of $A\to\mathcal M(\cont)$ weighted by $\psi\in\mathcal UA$ is $\cont\blacktriangleleft\psi$.
\item $(\phi,\psi)\in\mathcal M(\cont)$ is the colimit of $X\to\mathcal M(\cont)$ weighted by $\phi$ and the limit of $A\to\mathcal M(\cont)$ weighted by $\psi$.
\end{enumerate}
\end{corollary}

\begin{proof} We sketch the proof of items 1 and 2. 3 and 4 are variations of 2.
\begin{enumerate}
\item The Yoneda embedding $\overline\cont\to\mathcal DX$ factors through the fully faithful $\mathcal M(\cont)\to\mathcal DX$. Since the Yoneda embedding preserves limits, so does $\overline\cont\to\mathcal M(\cont)$. Dually, $\overline\cont\to\mathcal M(\cont)$ preserves colimits since the co-Yoneda embedding $\overline\cont\to\mathcal UA$ preserves colimits.
\item The colimit in $\mathcal M(\cont)$ is computed as the colimit in $\mathcal DX$ and then reflected into $\mathcal M(\cont)$ using the closure operator induced by the adjunction.
\end{enumerate}
\end{proof}

\subsection{The MacNeille Yoneda Lemma}\label{lem: mcyoneda}
\begin{gather*}
\mathcal M(\cont)(\overline x, \kappa) = \val{\kappa}(x)\\
\mathcal M(\cont)(\kappa, \overline a) = \descr{\kappa}(a)
%\mathcal M(\cont)(\overline x, (\phi,\psi)) = \phi(x)\\
%\mathcal M(\cont)(\overline a, (\phi,\psi)) = \psi(a)
\end{gather*}
\begin{proof} We prove the first statement, the second is dual.
\begin{align*}\mathcal M(\cont)(\overline x, \kappa) 
&=  \cont(x,-)\blacktriangleleft \descr{\kappa} 
&\text{ definitions \ref{sec:macneille-completion}, \ref{rmk:yoneda}}
\\&=  (X(-,x)\blacktriangleright \cont)\blacktriangleleft \descr{\kappa}  
& \text{Yoneda~\ref{rmk:yoneda}}
\\&=  X(-,x)\blacktriangleright (\cont \blacktriangleleft \descr{\kappa}  )
& \eqref{eq:adjunction}
\\&=  X(-,x)\blacktriangleright \val{\kappa} 
&\text{ definition \ref{sec:macneille-completion}}
\\&=   \val{\kappa}(x)
& \text{Yoneda}
\end{align*}
\end{proof}

\subsection{Algebraic MacNeille completion}

So far we have explicitly constructed the MacNeille completion as a certain $\Omega$-space. In the context of relations, the MacNeille completion also has an algebraic formulation, as the unique $\Delta_0$-completion of a relation. In this section we present an analoguous formulation in the quantale-enriched case.

\begin{definition}[Algebraic MacNeille completion]
    Let $\cont:X\looparrowright A$ be a context, and let $C$ be a complete and co-complete $\Omega$ space. Then $C$ is a MacNeille completion of $\cont$ if there exist functors $l:X\to C$ and $r:A\to C$ such that:
    \begin{enumerate}
        \item $C(lx,ra)=\cont(x,a)$, for every $x\in X$ and $a\in A$;
        \item for every $c\in C$, there exist $\varphi\in \mathcal DX$ and $\psi\in\mathcal UA$, such that ${\rm colim}_\varphi l=c={\rm lim}_\psi r $.
    \end{enumerate}
\end{definition}

\begin{lemma}
    Let $\cont:X\looparrowright A$ be a context, and $C$ a MacNeille completion of $\cont$. Then, if $\varphi\in\mathcal DX$, $\psi\in\mathcal UA$, then ${\rm colim }_\varphi l\equiv {\rm colim }_{\cont^{\downarrow}\cont^{\uparrow}\varphi} l$ and ${\rm lim }_\psi r\equiv {\rm lim }_{\cont^{\uparrow}\cont^{\downarrow}\psi} r$.
\end{lemma}
\begin{proof}
    Since colimits and limits are defined by the equations presented in Section \ref{sec:weighted-colimits}, it is enough to show that for every $c\in C$ $$C({\rm colim }_\varphi l,c)=C({\rm colim }_{\cont^{\downarrow}\cont^{\uparrow}\varphi} l,c) \qquad\text{and}\qquad C(c,{\rm lim }_\psi r)=C(c,{\rm lim }_{\cont^{\uparrow}\cont^{\downarrow}\psi} r).$$
    We show only the first equation the second being dual. Since $C$ is a MacNeille completion of $\cont$, we have that $c={\rm lim}_\psi r$, for some $\psi\in\mathcal UA$. Then we have:
    \begin{align*}
        C({\rm colim }_\varphi l,{\rm lim}_\psi r) 
        &=(\varphi\blacktriangleright C(l,r))\blacktriangleleft\psi &\text{definition}\\
        &=(\varphi\blacktriangleright\cont)\blacktriangleleft\psi & C(lx,ra)=\cont(x,a)\\
        &=\cont^{\uparrow}\varphi\blacktriangleleft\psi & \text{definition}\\
        &=\cont^{\uparrow}\cont^{\downarrow}\cont^{\uparrow}\varphi\blacktriangleleft\psi & \text{Section \ref{sec:idempotence}}\\
        &=(\cont^{\downarrow}\cont^{\uparrow}\varphi\blacktriangleright\cont)\blacktriangleleft\psi & \text{definition}\\
        &=(\cont^{\downarrow}\cont^{\uparrow}\varphi\blacktriangleright C(l,r))\blacktriangleleft\psi &C(lx,ra)=\cont(x,a)\\
        &= C({\rm colim }_{\cont^{\downarrow}\cont^{\uparrow}\varphi} l,{\rm lim}_{\psi}r) &\text{definition.}
    \end{align*}
\end{proof}

\begin{lemma}
    Let $\cont:X\looparrowright A$ be a context, and let $C$ be a MacNeille completion of $\cont$. Then ${\rm colim}_{\varphi}l\equiv {\rm lim}_\psi r$ for every $(\varphi,\psi)\in\mathcal M(\cont)$.
\end{lemma}
\begin{proof}
    We will show that $e\sqsubseteq C({\rm colim}_{\varphi}l,{\rm lim}_\psi r)$ and $e\sqsubseteq C({\rm lim}_\psi r,{\rm colim}_{\varphi}l)$.
 For the first inequality we have:
    \begin{align*}
        e&\sqsubseteq \psi\blacktriangleleft\psi & \\
        &=(\varphi\blacktriangleright\cont)\psi & (\varphi,\psi)\in\mathcal{M}(\cont)\\
        &=(\varphi\blacktriangleright C(l,r))\blacktriangleleft \psi & C(lx,ra)=\cont(x,a)\\
        &=C({\rm colim}_{\varphi}l,{\rm lim}_\psi r) & \text{definition.}
    \end{align*}
    For the second inequality, since $C$ is a MacNeille completion, there exists $\varphi'\in\mathcal DX$, such that ${\rm colim}_{\varphi'}l= {\rm lim}_\psi r$. Then 
    \begin{align*}
        e &\sqsubseteq C({\rm colim}_{\varphi'}l,{\rm lim}_\psi r) & \\
        & =\varphi'\blacktriangleright (C(l,r)\blacktriangleleft\psi) & \text{definition}\\
        & =\varphi'\blacktriangleright (\cont\blacktriangleleft\psi) &C(lx,ra)=\cont(x,a)\\
        &=\varphi'\blacktriangleright\varphi & (\varphi,\psi)\in\mathcal{M}(\cont).
    \end{align*}
    Since $e\sqsubseteq \varphi'\blacktriangleright\varphi$, is equivalent with $\varphi'\sqsubseteq\varphi$ we get:
    \begin{align*}
        e&\sqsubseteq C({\rm colim}_{\varphi}l,{\rm colim}_{\varphi}l) &\\
        &=\varphi\blacktriangleright C(l,{\rm colim}_{\varphi}l) & \text{definition}\\
        &\sqsubseteq\varphi'\blacktriangleright C(l,{\rm colim}_{\varphi}l) & \varphi'\sqsubseteq\varphi\\
        &=C({\rm colim}_{\varphi'}l,{\rm colim}_{\varphi}l) &\text{definition}\\
        &=C({\rm lim}_\psi r,{\rm colim}_{\varphi}l) & {\rm colim}_{\varphi'}l= {\rm lim}_\psi r.
    \end{align*}
\end{proof}

From the two lemmas above we immediately obtain the following:
\begin{corollary}
    Let $\cont:X\looparrowright A$ be a context, and let $C$ be a MacNeille completion of $\cont$. Then for every $c\in C$ there exists $(\varphi,\psi)\in\mathcal{M}(\cont)$ such that $${\rm colim}_{\varphi}l\equiv c\equiv {\rm lim}_{\psi}r.$$ 
\end{corollary}

\begin{theorem}
    The MacNeille completion of a context $\cont:X\looparrowright A$ is unique up to isomorphism.
\end{theorem}
\begin{proof}
    Let $C$ and $D$ be two MacNeille completions of $\cont$, witnessed by $l_C,r_C$ and $l_D,r_D$ respectively. For $c\in C$ and $d\in D$, let $\kappa_c,\kappa_d\in\mathcal{M}(\cont)$, be as per the corollary above. Define $f:C\to D$ by $f(c)={\rm colim}_{\val{\kappa_c}}l_D$ and $g:D\to C$ by $g(d)={\rm colim}_{\val{\kappa_d}}l_C$. We have that 
    \begin{align*}
        C(c_1,c_2)&=C({\rm colim}_{\val{\kappa_{c_1}}}l_C,{\rm lim}_{\descr{\kappa_{c_2}}}r_C) & \\
        &=\val{\kappa_{c_1}}\blacktriangleright C(l_C,r_C)\blacktriangleleft\descr{\kappa_{c_2}} & \\
        &=\val{\kappa_{c_1}}\blacktriangleright\cont\blacktriangleleft\descr{\kappa_{c_2}} &\\
        &=\val{\kappa_{c_1}}\blacktriangleright D(l_D,r_D)\blacktriangleleft\descr{\kappa_{c_2}} &\\
        &=D({\rm colim}_{\val{\kappa_{c_1}}}l_D,{\rm lim}_{\descr{\kappa_{c_2}}}r_D) & \\
        &=D(f(c_1),f(c_2)).
    \end{align*}
Likewise we can show that $D(d_1,d_2)=C(g(d_1),g(d_2))$. It's also immediate that $f$ and $g$ are essentially surjective. This concludes the proof.
\end{proof}
%\begin{ak}
%The MacNeille completion is unique up to isomorphism ... see \href{https://hackmd.io/@alexhkurz/BknC9boqJl}{here} for notes on this. 

%\end{ak}

\section{Canonical Extension}

Informally, the canonical extension of a quantale space is the MacNeille completion of the \emph{intermediate context} or \emph{intermediate level} that has upsets or filters in its ``lower layer'', downsets or ideals in its ``upper layer'' and the incidence relation given by intersection.

\begin{definition}
The canonical extension $C^\delta$ of a quantale space $C$, parameterised by subsets $F\subseteq\mathcal UC$ and $I\subseteq\mathcal DC$ is the MacNeille completion of the relation 
$$\cont:F\times I\to \Omega$$ 
given by
\begin{align}\label{eq:cont}
\cont (f,i) = \bigsqcup_c f(c)\cdot  i(c).
\end{align}
\end{definition}

\begin{remark}
In the definition, $F$ and $I$ are mere sets. As we have seen in Lemma~\ref{lem:mc-distance}, $F$ and $I$ inherit the structure of a quantale space from the MacNeille completion of $\cont$. $\cont(f,i)$ is the colimit of $f$ weighted by $i$ and the composition of relations $f\bullet i:1\looparrowright C \looparrowright 1$, see \eqref{eq:psi-phi}:
\begin{equation}\label{eq:Ifbulleti}
\cont(f,i)={\rm colim}_i f = f\bulletop i 
\end{equation}
\end{remark}

\begin{remark}\label{rmk:can-ext-1}
We parameterized the definition of canonical extension by $F\subseteq\mathcal UC$ and $I\subseteq\mathcal DC$. In the classical case $\Omega=2$, $F$ is the set of filters $f:C\to 2$  (finite meet preserving functions) and $I$ is the set of the ideals $i:C^o\to 2$ (finite join preserving functions). %The choice of these particular $F$ and $I$ are dictated by the application of canonical extensions to proving completeness results of various modal logics, but looks rather arbitrary from a category theoretic point of view. 
\end{remark}

In the following, to not burden our language, we will speak of filters and ideals as well as of finite limits and colimits, but keep in mind that $F$ and $I$ are paremeters.  We should also note that in the quantale-enriched case finite (co)limits refer to finite \emph{weighted} (co)limits, that is, a filter $f: C\to\Omega$ preserves finite meets and powers while an ideal $i:C^o\to\Omega$ preserves finite joins and tensors. An assumption we always make is that $F$ and $I$ contain all representable (co)presheaves. 

We summarize these definitions and conventions in the following diagram.
% https://q.uiver.app/#q=WzAsNixbMCwwLCJ7XFxjYWwgRFV9QyJdLFsyLDIsIkNeXFxkZWx0YSJdLFs0LDAsIntcXGNhbCBVRH1DIl0sWzAsNCwiXFxtYXRoY2FsIFUDIl0sWzQsNCwiXFxtYXRoY2FsIERDIl0sWzIsNSwiQyJdLFszLDBdLFs0LDJdLFszLDQsIlxccm0gSSIsMix7InN0eWxlIjp7ImJvZHkiOnsibmFtZSI6ImRvdHRlZCJ9fX1dLFszLDEsImZcXG1hcHN0byBcXG92ZXJsaW5lIGYiLDFdLFs0LDEsImlcXG1hcHN0byBcXG92ZXJsaW5lIGkiLDFdLFswLDIsIi1cXGJsYWNrdHJpYW5nbGVyaWdodCBcXHJtIEkiLDAseyJjdXJ2ZSI6LTJ9XSxbMiwwLCJ7XFxybSBJfVxcYmxhY2t0cmlhbmdsZWxlZnQgLSIsMCx7ImN1cnZlIjotMn1dLFsxLDAsIiIsMCx7InN0eWxlIjp7InRhaWwiOnsibmFtZSI6Imhvb2siLCJzaWRlIjoiYm90dG9tIn19fV0sWzEsMiwiIiwwLHsic3R5bGUiOnsidGFpbCI6eyJuYW1lIjoiaG9vayIsInNpZGUiOiJ0b3AifX19XSxbMCwxLCIiLDAseyJjdXJ2ZSI6Miwic3R5bGUiOnsiYm9keSI6eyJuYW1lIjoiZGFzaGVkIn19fV0sWzAsMiwiXFxib3QiLDEseyJzdHlsZSI6eyJib2R5Ijp7Im5hbWUiOiJub25lIn0sImhlYWQiOnsibmFtZSI6Im5vbmUifX19XSxbMCwxLCJcXGRhc2h2IiwxLHsiY3VydmUiOjEsInN0eWxlIjp7ImJvZHkiOnsibmFtZSI6Im5vbmUifSwiaGVhZCI6eyJuYW1lIjoibm9uZSJ9fX1dLFsyLDEsIiIsMSx7ImN1cnZlIjotMiwic3R5bGUiOnsiYm9keSI6eyJuYW1lIjoiZGFzaGVkIn19fV0sWzEsMiwiXFxkYXNodiIsMSx7ImN1cnZlIjoxLCJzdHlsZSI6eyJib2R5Ijp7Im5hbWUiOiJub25lIn0sImhlYWQiOnsibmFtZSI6Im5vbmUifX19XSxbNSwzXSxbNSw0XSxbNSwxLCJbLV0iLDIseyJjdXJ2ZSI6MX1dXQ==
\[\begin{tikzcd}
	{{\mathcal DF}} &&&& {{\mathcal UI}} \\
	\\
	&& {C^\delta} \\
	\\
	{\mathcal UC\supseteq F} &&&& {I\subseteq \mathcal DC} \\
	&& C
	\arrow[from=5-1, to=1-1]
	\arrow[from=5-5, to=1-5]
	\arrow["{\rm I}"', dotted, from=5-1, to=5-5]
	\arrow["{f\mapsto \overline f}"{description}, from=5-1, to=3-3]
	\arrow["{i\mapsto \overline i}"{description}, from=5-5, to=3-3]
	\arrow["{-\blacktriangleright \rm I}", curve={height=-12pt}, from=1-1, to=1-5]
	\arrow["{{\rm I}\blacktriangleleft -}", curve={height=-12pt}, from=1-5, to=1-1]
	\arrow[hook', from=3-3, to=1-1]
	\arrow[hook, from=3-3, to=1-5]
	\arrow[curve={height=16pt}, dashed, from=1-1, to=3-3]
	\arrow["\bot"{description}, draw=none, from=1-1, to=1-5]
	\arrow["\dashv"{description}, curve={height=8pt}, draw=none, from=1-1, to=3-3]
	\arrow[curve={height=-16pt}, dashed, from=1-5, to=3-3]
	\arrow["\dashv"{description}, curve={height=8pt}, draw=none, from=3-3, to=1-5]
	\arrow[from=6-3, to=5-1]
	\arrow[from=6-3, to=5-5]
	\arrow["{[-]}"', curve={height=16pt}, from=6-3, to=3-3]
\end{tikzcd}\]
%Typically, $F$ and $I$ are  subsets of $\mathcal UC$ and $\mathcal DC$ containing the representable (co)presheaves. 
%The paradigmatic example is the set of all ``weighted'' filters $f$ and ``weighted'' ideals $i$.

\begin{remark}\label{rmk:yonedaI} \eqref{eq:identity} and \eqref{eq:cont} imply the following Yoneda lemma.
\begin{gather*}
\cont(C(c,-),i)=i(c), \\
\cont(f,C(-,c))=f(c).
\end{gather*}
\end{remark}

\begin{comment}
\begin{ak}
\begin{example} \label{ex:canonical-extension} \ 

Boolean algebras \dots

The canonical extension of a Boolean algebra $A$ can be described as the powerset of ultrafilters on $A$. 

(filters = compactness)

To a set of ultrafilters associate the filter of the intersections of the ultrafilters. Such a filter is a fixed point of the adjunction.  (Why?) Conversely, every filter (?) is the intersection of the ultrafilters containing it. 

If $F$ and $I$ are all upsets / downsets all (co)limits are destroyed ... virtual adjoints ... PDL and DEL ... 

If $F$ and $I$ are all (co)presheaves we get the MacNeille completion: If $f=C(c,-)$ and $i=C(-,d)$ then $\cont(f,i)=C(c,d)$

Distributive lattices \dots

Posets \dots

More examples \ldots 

If $C$ is a poset and $F$ is the set of presheaves and $I$ is the set of co-presheaves and $\cont(C(c,-),C(-,d))=C(c,d)$ then $C^\delta$ is the MacNeille completion of $C$.

Proof: $\mathcal M(\cont)\to C^\delta$ ... 

\end{example}
\end{ak}
\end{comment}
\begin{table}[h]
	\centering
	\begin{tabular}{|c|c|}
	\hline
	\textbf{Posets} & \textbf{Categories} \\
	\hline
	$c\le d$ & $C(c,d)$ \\
	\hline
	${\uparrow} c$ & $C(c,-)$ \\
	\hline
	${\downarrow} c$ & $C(-,c)$ \\
	\hline
	$f$ filter & $f:1\looparrowright C$ preserves finite limits \\
	\hline
	$i$ ideal & $i:C\looparrowright 1$ preserves finite colimits \\
	\hline
	$c \in f$, $c\in i$ & $f(c), i(c)$ \\
	\hline
	$f \cap i\not=\emptyset$ & $f\bullet i$ \\
	\hline
	closed element & ${\rm lim}_f\,[-]$ for some $f$\\
	\hline
	open element & ${\rm colim}_i\,[-]$ for some $i$\\
	\hline
	\end{tabular}
	\caption{Correspondence between poset and category notations}
	\label{tab:poset-cat}
\end{table}

\subsection{The Embedding into the Canonical Extension} The embedding $C\to C^\delta$ arises in two different ways, composing the Yoneda embedding with the MacNeille embedding, see Remark~\ref{rmk:yoneda}, as
\begin{gather*}
C
\ \longrightarrow \ \mathcal UC 
\ \longrightarrow \ {\cal DU}C
\ \longrightarrow \ {\cal UD}C 
\\
c\ \mapsto \ C(c,-)
\ \mapsto \ \mathcal UC(-,C(c,-))
\ \mapsto \ \cont(C(c,-),-)
\end{gather*}
and
\begin{gather*}
C
\ \longrightarrow \  \mathcal DC 
\ \longrightarrow \  {\cal UD}C
\ \longrightarrow \  {\cal DU}C\\
c\ \mapsto \ C(-,c)
\ \mapsto \ \mathcal DC(C(-,c),-)
\ \mapsto \ \cont(-,C(-,c))
\end{gather*}
To show that they are equal, recalling the definition of MacNeille completion \ref{sec:macneille-completion}, it suffices to show  
\begin{gather*}
\cont(C(c,-),-) \ = \ \mathcal DC(C(-,c),-)
\\
\cont (-,C(-,c)) \ = \ \mathcal UC(-,C(c,-))  
\end{gather*}
which follows from using \ref{rmk:yonedaI} on the left and Yoneda on the right of the equations. This also implies the following.f

\begin{remark}
    The ``double Yoneda" embeddings given by $c\mapsto \mathcal DC(C(-,c),-)$ and $c\mapsto \mathcal UC(-,C(c,-))$ factor through $C\to C^\delta$. Since they are the composition of a covariant and a contravariant Yoneda embedding, they do not preserve limits and do not preserve colimits.
\end{remark}

\begin{definition}
The inclusion $[-]:C\to C^\delta$ is defined by $\val{[c]}(f)=f(c)$ and $\descr{[c]}(i)=i(c)$.
\end{definition}

\begin{remark}
Given $G:X\to C$ we may write $[G]$ for $[-]\circ G$.
\end{remark}

\begin{lemma}\label{lem:embedding-preserves-limits]}
 Let $G:X\to C$, $g\in\mathcal UX$, $j\in\mathcal DX$, $\phi\in\mathcal DF$, and  $\psi\in\mathcal UI$. Assume that co-presheaves in $F$ preserve $g$-limits and that presheaves in $I$ preserve $j$-colimits. Then 
\begin{align*}
\phi\blacktriangleright [{\rm lim}_g\, G] &= \phi\blacktriangleright {\rm lim}_g\, [G] \\[0.5ex]
[{\rm colim}_j\,G]\blacktriangleleft \psi &= {\rm colim}_j\,[G]\blacktriangleleft \psi
 \end{align*}
\end{lemma}
\begin{proof} The first claim is proved as follows.
\begin{align*}
\phi \blacktriangleright [{\rm lim}_g G]
&= \bigsqcap_{f\in F} \phi(f) \rhd [{\rm lim}_g G](f)
& \text{def of } \blacktriangleright
\\
&= \bigsqcap_{f\in F} \phi(f) \rhd f({\rm lim}_g G)
& \text{def of } [-]
\\
&= \bigsqcap_{f\in F} \phi(f) \rhd {\rm lim}_g (f\circ G)
& f \text{ preserves $g$-limits} 
\\
&= \bigsqcap_{f\in F} \phi(f) \rhd {\rm lim}_g [G](f)
& \text{def of } [-]
\\
&= \phi \blacktriangleright {\rm lim}_g\, [G] 
& \text{def of } \blacktriangleright
\end{align*}
The second statement is dual to the first.
\end{proof}

As a corollary of the lemma, we want to say that $[-]:C\to C^\delta$ preserves those limits that the co-presheaves in $F$ preserve and those colimits that the presheaves in $I$ preserve. To this end, we use the notion of class of weights from Kelly and Schmitt~\cite{Kelly-Schmitt}.

\begin{proposition} Let $\Phi$ and $\Psi$ be a classes of weights. If co-presheaves in $F$ preserve $\Psi$-limits, then $[-]:C\to C^\delta$ preserves $\Psi$-limits. If presheaves in $I$ preserve $\Phi$-colimits, then $[-]:C\to C^\delta$ preserves $\Phi$-colimits.
\end{proposition}

\begin{proof} We show that $[\lim_g G]=\lim_g [G]$ for all $(g:1\looparrowright X)\in\Psi$ and $G:X\to C$.
\begin{align*}
C^\delta((\phi,\psi),[{\rm lim}_g G])
&= \phi \blacktriangleright [{\rm lim}_g G]
& \text{def of } C^\delta
\\
&= \phi \blacktriangleright {\rm lim}_g\, [G] 
& \text{Lemma \ref{lem:embedding-preserves-limits]}}
\\
&= C^\delta((\phi,\psi), {\rm lim}_g\, [G] ) 
& \text{def of } C^\delta
\end{align*}
The second statement is dual to the first.
\end{proof}

\begin{proposition}
    Let $f\in F$ and $i\in I$. Then $C^{\delta}(\kappa,{\rm lim}_f \,[-])=\descr{\kappa}\blacktriangleleft \cont(f,-)$ and  $C^{\delta}({\rm colim}_i\,[-],\kappa)=\cont(-,i)\blacktriangleright\val{\kappa} $.
\end{proposition}

\begin{proof} 
The proof uses the right-hand side of Lemma~\ref{lem:double-blacktriangle} and instantiates $T$ with $\descr{\kappa}: 1\looparrowright I$
 and $S$ with $[-]: C\looparrowright I$
 and $S'$ with $f:1\looparrowright C$

\begin{align*}
C^{\delta}(\kappa,{\rm lim}_f [-]) &= C^{\delta}(\kappa, [-]) \blacktriangleleft f 
& \text{def of limits} 
\\
&= (\descr{\kappa} \blacktriangleleft [-]) \blacktriangleleft f & \text{def of } C^\delta\\
&= \descr{\kappa} \blacktriangleleft (f\bullet [-]) & \text{Lemma \ref{lem:double-blacktriangle}}\\
&= \descr{\kappa} \blacktriangleleft \cont(f,-). & \text{def of $\cont$ \eqref{eq:Ifbulleti}}
\end{align*}
The second equality is dual to the first.
\end{proof}

\begin{corollary}\label{cor: fi lims}
Let $\kappa_f={\rm lim}_f\, [-]$ and $\kappa_i= {\rm colim}_i[-]$ then $$\descr{\kappa_f}(i)=\val{\kappa_i}(f)=\cont(f,i).$$ In particular, $\lim_f [-]=\overline{f}$ and ${\rm colim}_i[-]=\overline{i}$.
\end{corollary}

\begin{theorem}[Compactness]\label{th:comp}
Let $f\in F$ and $i\in I$. Then $C^\delta(\lim_f [-],{\rm colim}_{i}[-])=\cont (f,i).$
\end{theorem}
\begin{proof}The theorem follows immediately from item 1 of  Lemma \ref{lem:mc-distance}, since by Corollary \ref{cor: fi lims} above we have that $\lim_f[-]=\overline{f}$ and ${\rm colim}_i[-]=\overline{i}$.
%For the left to right direction we have:
%       \begin{align*}
%            C^\delta(\lim_f [-],{\rm colim}_{i}[-])
%            &= \cont(f,-)https://www.overleaf.com/project/65d26ddf7d12ea30fa84645b\blacktriangleleft (\cont(-,i)\blacktriangleright \cont) &\\
%            &\leq \cont(f,i)\lhd (\cont(-,i) \blacktriangleright \cont(-,i)) & \text{definition}\\
%            &\leq \cont(f,i)\lhd e & \\
%            &=\cont(f,i).
%       \end{align*}
\end{proof}
\begin{theorem}[Density]\label{th:d1dense}
   Every $\kappa\in C^\delta$ is the colimit of a limit of $C$ and the limit of a colimit of $C$.
\end{theorem}
\begin{proof}
    By Corrolary \ref{cor: limits in mc} every element of $C^\delta$ is a limit of $\mathcal{U}C$ and a colimit of $\mathcal{D}C$. At the same time, by Corollary \ref{cor: fi lims} every element of $\mathcal{U}C$ is a colimit of $C$ and every element of $\mathcal{D}C$ is a limit of $C$.
\end{proof}

%\begin{ak}
%... uniqueness ... see the \href{https://hackmd.io/@alexhkurz/Bkt9GZqJeg}{this markdown}
%\end{ak}

We now turn our attention to extending a functor $G:C\to D$ between $\Omega$-categories via the intermediate level to functors $G^\sigma,G^\pi:C^\delta\to D^\delta$ on canonical extensions.

% https://q.uiver.app/#q=WzAsMTAsWzIsMCwiQyJdLFs0LDAsIkQiXSxbNCw0LCJDXlxcZGVsdGEiXSxbNiw0LCJEXlxcZGVsdGEiXSxbMCw0LCJDXlxcZGVsdGEiXSxbMiw0LCJEXlxcZGVsdGEiXSxbMiwyLCJcXG1hdGhjYWwgREQiXSxbNCwyLCJcXG1hdGhjYWwgVUMiXSxbMCwyLCJcXG1hdGhjYWwgREMiXSxbNiwyLCJcXG1hdGhjYWwgVUQiXSxbMCwxLCJHIl0sWzMsMiwiR15yIiwwLHsiY3VydmUiOi0yfV0sWzIsMywiR15cXHNpZ21hIiwwLHsiY3VydmUiOi0yfV0sWzUsNCwiR15sIiwwLHsiY3VydmUiOi0yfV0sWzQsNSwiR15cXHBpIiwwLHsiY3VydmUiOi0yfV0sWzQsNSwiXFx0b3AiLDEseyJzdHlsZSI6eyJib2R5Ijp7Im5hbWUiOiJub25lIn0sImhlYWQiOnsibmFtZSI6Im5vbmUifX19XSxbMiwzLCJcXGJvdCIsMSx7InN0eWxlIjp7ImJvZHkiOnsibmFtZSI6Im5vbmUifSwiaGVhZCI6eyJuYW1lIjoiZW5kIn19fV0sWzYsOCwiR15yIiwwLHsiY3VydmUiOi0yfV0sWzksNywiR15sIiwwLHsiY3VydmUiOi0yfV0sWzgsNiwiR15cXHBpIiwwLHsiY3VydmUiOi0yfV0sWzcsOSwiR15cXHNpZ21hIiwwLHsiY3VydmUiOi0yfV0sWzgsNiwiXFxib3QiLDEseyJzdHlsZSI6eyJib2R5Ijp7Im5hbWUiOiJub25lIn0sImhlYWQiOnsibmFtZSI6Im5vbmUifX19XV0=
\[\begin{tikzcd}
	&& C && D \\
	\\
	{\mathcal DC} && {\mathcal DD} && {\mathcal UC} && {\mathcal UD} \\
	\\
	{C^\delta} && {D^\delta} && {C^\delta} && {D^\delta}
	\arrow["G", from=1-3, to=1-5]
	\arrow["{G^r}", curve={height=-12pt}, from=5-7, to=5-5]
	\arrow["{G^\sigma}", curve={height=-12pt}, from=5-5, to=5-7]
	\arrow["{G^l}", curve={height=-12pt}, from=5-3, to=5-1]
	\arrow["{G^\pi}", curve={height=-12pt}, from=5-1, to=5-3]
	\arrow["\top"{description}, draw=none, from=5-1, to=5-3]
	\arrow["\bot"{description}, draw=none, from=5-5, to=5-7]
	\arrow["{G_r}", curve={height=-12pt}, from=3-3, to=3-1]
	\arrow["{G_l}", curve={height=-12pt}, from=3-7, to=3-5]
	\arrow["{G_\pi}", curve={height=-12pt}, from=3-1, to=3-3]
	\arrow["{G_\sigma}", curve={height=-12pt}, from=3-5, to=3-7]
	\arrow["\bot"{description}, draw=none, from=3-1, to=3-3]
	\arrow["\top"{description}, draw=none, from=3-5, to=3-7]
\end{tikzcd}\]

\begin{remark}\label{rmk:can-ext-2}
    The $\sigma$- and $\pi$-extension of functors will play a crucial role (to be pursued in future work) in the application of canonical extenions to the logic of relational structures on quantale-enriched categories.  In these applications, there will be logical connectives stemming from the weighted (co)limits of $\Omega$ (meets, joins, tensor, power) as well as additional modal operators corresponding to operations (functors) on modal algebras. 
    
    Canonical extensions offer a modular and generic method of proving completeness of \emph{finitary} modal logics with respect to various relational semantics. The method hinges on extending functora $f:\mathbb{A}\to\mathbb{B}$ that preserve \emph{finite} finite weighted (co)limits to maps on the canonical extensions $f^\sigma,f^\pi:\mathbb{A}^\delta\to\mathbb{B}^\delta$ that preserve arbitrary (co)limits and, hence, have adjoints. 
    
    Indeed, on the syntactic side, modalities such as $\Box$ and $\lozenge$ can be seen as monotone operations that preserve finite meets and joins respectively on the Lindenbaum-Tarski algebra of the logic. On the semantic side, such modalities arise from relations on sets, and therefore preserve arbitrary meets and joins and are right and left adjoints respectively.  Then completeness of such modal logics is shown via canonicity: $\Box^\pi$ and $\lozenge^\sigma$ do have adjoints and therefore can correspond to relations on the dual space of the canonical extension. 
\end{remark}

\begin{definition}
%Let $\Phi, \Psi$ be classes of weights such that $\psi\circ G\in \Psi_D$ for every $\psi\in \Psi_C$. Then, given $f\in\mathcal UD$ and $i\in\mathcal  D D$, we define $G_{l}(f),G_{r}(i):C\to\Omega$, as $$G_{l}(f)(c)=f(G(c)) \quad\quad G_{r}(i)(c)=i(G(c)).$$
Given a functor $G:C\to D$ between $\Omega$-spaces, we define by precomposition the functors 
% https://q.uiver.app/#q=WzAsNCxbMiwwLCJcXG1hdGhjYWwgREQiXSxbNCwwLCJcXG1hdGhjYWwgVUMiXSxbMCwwLCJcXG1hdGhjYWwgREMiXSxbNiwwLCJcXG1hdGhjYWwgVUQiXSxbMCwyLCJHX3IiXSxbMywxLCJHX2wiXV0=
\[\begin{tikzcd}
	{\mathcal DC} && {\mathcal DD} && {\mathcal UC} && {\mathcal UD}
	\arrow["{G_r}", from=1-3, to=1-1]
	\arrow["{G_l}", from=1-7, to=1-5]
\end{tikzcd}\]
In detail, given $f\in\mathcal UD$ and $i\in\mathcal  D D$, we define $G_{l}(f)(c)=f(G(c))$ and $G_{r}(i)(c)=i(G(c))$.
\end{definition}

We next show that if $G$ preserves finite limits and colimits, then $G_l$ and $G_r$ restrict to filters and ideals.

\begin{proposition}
Let $G:C\to D$.
Let $\Phi, \Psi$ be classes of weights. 
Let $f\in\mathcal UD$ preserve $\Psi$-limits and let $i\in\mathcal DD$ preserve $\Phi$-colimits.
If $G$ preserves $\Psi$-limits, then so does $G_{l}(f)$ and if $G$ preserves $\Phi$-colimits, then so does $G_{r}(i)$.
\end{proposition}
\begin{proof}
Immediate from the definition of $G_l$ and $G_r$ as precomposition with $G$. 
\end{proof}

We now extend $G_l$ and $G_r$ from the intermediate level to canonical extensions. 
\[\begin{tikzcd}
	{C^\delta} && {D^\delta} && {C^\delta} && {D^\delta}
	\arrow["{G^r}", from=1-3, to=1-1]
	\arrow["{G^l}", from=1-7, to=1-5]
\end{tikzcd}\]
Recall that for $\kappa\in D^\delta$, we have $\val{\kappa}\in\mathcal DF_D$ and $\descr{\kappa}\in\mathcal U I_D$.

\begin{definition}
    We define $G^{l},G^{r}:D^\delta\to C^\delta$ as $$G^{l}(\kappa)={\rm colim}_{\val{\kappa}}\, \overline G_l 
    \quad\quad \quad 
    G^{r}(\kappa)={\rm lim}_{\descr{\kappa}} \,\overline G_{r}.$$
\end{definition}

\begin{remark}
To show that $G^l$ is functorial note that $\val{\kappa}\sqsubseteq \val{\kappa'}\blacktriangleleft D^\delta(\kappa,\kappa')$. Then we have:
\begin{align*}
C^\delta({\rm colim}_{\val{\kappa}}\, \overline G_l,{\rm colim}_{\val{\kappa'}}\, \overline G_l) 
& = \val{\kappa}\blacktriangleright C^{\delta}(\overline G_{l},{\rm colim}_{\val{\kappa'}}\, \overline G_l) 
& \text{def of co-limits} 
\\
&\sqsupseteq (\val{\kappa'}\blacktriangleleft D^\delta(\kappa,\kappa'))\blacktriangleright C^{\delta}(\overline G_{l},{\rm colim}_{\val{\kappa'}}\, \overline G_l) 
& \text{tonicity} 
\\     
&\sqsupseteq (\val{\kappa'}\blacktriangleleft D^\delta(\kappa,\kappa'))\blacktriangleright \val{\kappa'} & \text{tonicity}\\
&\sqsupseteq D^\delta(\kappa,\kappa') & \text{quantale identity}
\end{align*}
The functoriality of $G^{r}$ is dual.
\end{remark}

The next proposition shows that $G^{l}$ and $G^{r}$ behave like adjoints on the image of the embeddings $[-]:C\to C^\delta$ and $[-]:D\to D^\delta$.  
%(However, it is not always the case that this adjointness can be extended to the canonical extensions beyond the images.)

\begin{proposition}
    For every $c\in C_0$, $C^{\delta}(G^{l}([G(c)]),[c])\sqsupseteq e$ and $C^{\delta}([c], G^{r}([G(c)]))\sqsupseteq e$.
\end{proposition}
\begin{proof}
    We have
    \begin{align*}
        C^\delta(G^{l}([G(c)]),[c]) &= \val{[G(c)]}\blacktriangleright C^{\delta}(G_{l},[c]) &\text{def of $G^{l}$}\\
        &=\val{[G(c)]}\blacktriangleright\val{[c]}(G_{l}) &\text{Lemma \ref{lem: mcyoneda}}\\
        &=\val{[G(c)]}\blacktriangleright \val{[G(c)]} &\text{def of $[c]$}\\
        &\sqsupseteq e.
    \end{align*}
    The second equation is dual.
\end{proof}

We next define $\sigma$- and $\pi$-extensions on the intermediate level:

% https://q.uiver.app/#q=WzAsNCxbMiwwLCJcXG1hdGhjYWwgREQiXSxbNCwwLCJcXG1hdGhjYWwgVUMiXSxbMCwwLCJcXG1hdGhjYWwgREMiXSxbNiwwLCJcXG1hdGhjYWwgVUQiXSxbMCwyLCJHX3IiLDAseyJjdXJ2ZSI6LTJ9XSxbMywxLCJHX2wiLDAseyJjdXJ2ZSI6LTJ9XSxbMSwzLCJHX1xcc2lnbWEiLDAseyJjdXJ2ZSI6LTJ9XSxbMSwzLCJcXHRvcCIsMSx7InN0eWxlIjp7ImJvZHkiOnsibmFtZSI6Im5vbmUifSwiaGVhZCI6eyJuYW1lIjoibm9uZSJ9fX1dLFsyLDAsIkdfXFxwaSIsMCx7Im9mZnNldCI6LTEsImN1cnZlIjotMX1dLFsyLDAsIlxcYm90IiwxLHsic3R5bGUiOnsiYm9keSI6eyJuYW1lIjoiZW5kIn0sImhlYWQiOnsibmFtZSI6Im5vbmUifX19fV1d
\[\begin{tikzcd}
	{\mathcal DC} && {\mathcal DD} && {\mathcal UC} && {\mathcal UD}
	\arrow["{G_r}", curve={height=-12pt}, from=1-3, to=1-1]
	\arrow["{G_l}", curve={height=-12pt}, from=1-7, to=1-5]
	\arrow["{G_\sigma}", curve={height=-12pt}, from=1-5, to=1-7]
	\arrow["\top"{description}, draw=none, from=1-5, to=1-7]
	\arrow["{G_\pi}", shift left, curve={height=-6pt}, from=1-1, to=1-3]
	\arrow["\bot"{description}, draw=none, from=1-1, to=1-3]
\end{tikzcd}\]

\begin{definition}
Given $G:C\to D$, $i\in\mathcal DC$ and $f\in\mathcal UC$, we define 
$$G_{\pi}(i) = {\rm colim}_i\, D(-,G)
\quad\quad\quad 
G_{\sigma}(f)= {\rm lim}_f\, D(G,-).$$
\end{definition}

\begin{remark}
The adjointness relations indicated in the diagrams are well-known. In our notation, the proof reads as follows.
\begin{align*}
\mathcal DD(G_\pi(i),j)
&= \mathcal DD({\rm colim}_i\, D(-,G),j)
& \text{def of } G_\pi\\
&= i\blacktriangleright \mathcal DD(D(-,G),j)
& \text{def of } {\rm colim}\\
&= i\blacktriangleright (j\circ G)
& \text{Yoneda}
\\
&= \mathcal DC(i, j\circ G)
& \text{def of } \mathcal D
\end{align*}
\end{remark}

%\begin{proposition}\label{prop:theyareatleastpresheaves}
%     $G_{\pi}(i)\in \mathcal{D}D$ and $G_{\sigma}(f)\in\mathcal{U}D$.
%\end{proposition}
%\begin{proof}
%    \begin{align*}
%    &D(d_1,d_2)\cdot D(d_2,G(c_1))\cdot i(c_1)\sqsubseteq D(d_1,G(-))\bullet i\\
%    \iff & D(d_1,d_2)\cdot D(d_2,G(-))\bullet i\sqsubseteq D(d_1,G(-))\bullet i\\
%     \iff & D(d_1,d_2)\cdot G_{\pi}(i)(d_2)\sqsubseteq G_{\pi}(i)(d_1)\\
%    \iff & D(d_1,d_2)\sqsubseteq G_{\pi}(i)(d_1)\lhd G_{\pi}(i)(d_2).
%\end{align*}
%\end{proof}

The next proposition will be needed to prove the extension theorem.

\begin{proposition}\label{prop:adj}
    Let $G:C\to D$, then $$G_{l}(f)\bullet i=f\bullet G_{\pi}(i)\quad\quad f\bullet G_{r}(i)=G_{\sigma}(f)\bullet i.$$
\end{proposition}
\begin{proof}
    We have: 
    \begin{align*}
        & f(G(c))= f\bullet D(-,G(c)) &\\
        \Longrightarrow & G_{l}(f)\bullet i = (f\bullet D (-, G (-))\bullet i &\\
        \iff & G_{l}(f)\bullet i = f\bullet(D(-,G(-)))\bullet i) &\\
        \iff & G_{l}(f)\bullet i = f\bullet G_{\pi}(i). &
    \end{align*}
\end{proof}

Finally, we can define the $\sigma$- and $\pi$-extensions of $G:C\to D$ on the canonical extensions.

\begin{definition}
We define functors $G^{\pi},G^{\sigma}:C^{\delta}\to D^{\delta}$ 
$$G^{\pi}(\kappa)={\rm lim}_{\descr{\kappa}}\, \overline G_{\pi} 
\quad\quad\quad
G^{\sigma}(\kappa)={\rm colim}_{\val{\kappa}}\overline G_\sigma
$$
\end{definition}

\begin{theorem}
\begin{enumerate}
    \item If $G_{l}(f)\in F_C$ and $G_{\pi}(i)\in I_D$ for every $i\in I_C$ and $f\in F_D$, then $G^{l}\dashv G^\pi$.
    \item If $G_{\sigma}(f)\in F_D$ and $G_{r}(i)\in I_C$ for every $i\in I_D$ and $f\in F_C$ then $G^{\sigma}\dashv G^{r}$.
\end{enumerate}
\end{theorem}
\begin{proof}
    We only show item 1, the second item being dual. By Theorem \ref{th:d1dense}, it is enough to show that $C^{\delta}(G^{l}(\overline{f}),\overline{i})= D^{\delta}(\overline{f},G^{\pi}(\overline{i}))$. Since $G_{\pi}(i)\in I_D$ and $G_{l}(f)\in F_C$, by Theorem \ref{th:comp}, it suffices to show that $\cont(G_{l}(f),i)=\cont(f,G_{\pi}(i))$. But this is the content of Proposition \ref{prop:adj}.
\end{proof}

\begin{remark}
    If $F_D=\mathcal{U}D$ or $I_D=\mathcal{D}D$, then 
    %by Proposition~\ref{prop:theyareatleastpresheaves} 
    $G_{\pi}(i)\in I_D$ and $C_{\delta}(f)\in F_D$. If $I$, is defined by ideals closed under $\Phi$-colimits, then $G_{\pi}(i)\in I_D$ is equivalent to the fact that for every $\phi\in \Phi$,  $(\phi\blacktriangleright D(X,G))\bullet i=\phi\blacktriangleright (D(X,G)\bullet i)$, where $\phi:X\to D$, which does not always hold. However, notice that $G^{\pi}, G^{\sigma}:C^\delta\to D^\delta$ are definable regardless of whether $G_{\pi}(i)\in I_D$ and $G_{\sigma}(f)\in F_D$ or not.
\end{remark}

\section{Conclusion}

We generalized the canonical extension construction from lattices and posets to quantale-enriched categories and developed the basic theory of these canonical extensions. On the category theoretic side this adds to a line of developement that started with Lawvere \cite{Lawvere:metric} and continued through, for example, \cite{Stubbe05a,Shen-Zhang,Willerton:tight-spans,Fujii}. On the logic side, our work is a  stepping stone as part of the larger project of developing the logical foundations of categorization theory \cite{CFPPTW16,TarkPaper2017}. In particular, our result is the theoretical groundwork needed for the development of the duality theory of quantale-enriched categories with many-valued polarities and the introduction of logical systems with modal operations and consequence relations weighted over a quantale, and the development of their proof theory and relational semantics. The work in the present article will allow the systematic development of classes of such logical systems, and allow to develop tools from the theory of algorithmic canonicity in the context of quantale enriched categories.  

%Future directions. The next steps in the development of the theory of canonical extensions in the categorical setting would be the study of extensions of functors of higher arity. This would also provide the theoretical foundation for the development of substructural logics with fuzzy consequence relations.

%\bibliographystyle{unsrtnat} 
\bibliographystyle{plainnat} 
\bibliography{bib}

\end{document}